\documentclass{amsart}

\usepackage{amsthm, amssymb, amsmath}
\usepackage{a4wide}
\usepackage[dvipsnames]{xcolor}
\usepackage[english]{babel}
\usepackage{graphicx, tikz, pgfplots}
\usepackage{subfigure}
\usepackage{url}

\newtheorem{theorem}{Theorem}[section]
\newtheorem{prop}{Proposition}[section]
\newtheorem{lemma}{Lemma}[section]
\newtheorem{cor}{Corollary}[section]

\newtheorem{remark}{Remark}[section]
\theoremstyle{definition}
\newtheorem{defn}{Definition}[section]
\newtheorem{ex}{Example}[section]
\newtheorem{main}{Theorem}

\newcommand{\La}{\mathcal L_{\eta}}

\begin{document}

\title{Invariant measures, matching and the frequency of 0 for signed binary expansions}
\author{Karma Dajani and Charlene Kalle}
\address[Karma Dajani]{Department of Mathematics, Utrecht University, P.O.~Box 80010, 3508TA Utrecht, the Netherlands}
\email[Karma Dajani]{k.dajani1@uu.nl}
\address[Charlene Kalle]{Mathematisch Instituut, Leiden University, Niels Bohrweg 1, 2333CA Leiden, The Netherlands}
\email[Charlene Kalle]{kallecccj@math.leidenuniv.nl}

\thanks{The second author was partially supported by the NWO Veni-grant 639.031.140. We thank Niels Langeveld for producing Figure~\ref{f:maximalinterval}}

\date{Version of \today}


\maketitle
 
\begin{abstract}
We introduce a parametrised family of maps $\{S_{\eta}\}_{\eta \in [1,2]}$, called symmetric doubling maps, defined on $[-1,1]$ by $S_\eta (x)=2x-d\eta$, where $d\in \{-1,0,1 \}$. Each map $S_\eta$ generates binary expansions with digits $-1$, 0 and 1. We study the frequency of the digit 0 in typical expansions as a function of the parameter $\eta$. The transformations $S_\eta$ have a natural ergodic invariant measure $\mu_\eta$ that is absolutely continuous with respect to Lebesgue measure. The frequency of the digit 0 is related to the measure $\mu_{\eta}([-\frac12,\frac12])$ by the Ergodic Theorem. We show that the density of $\mu_\eta$ is piecewise smooth except for a set of parameters of zero Lebesgue measure and full Hausdorff dimension and give a full description of the structure of the maximal parameter intervals on which the density is piecewise smooth. We give an explicit formula for the frequency of the digit 0 in typical signed binary expansions on each of these parameter intervals and show that this frequency depends continuously on the parameter $\eta$. Moreover, it takes the value $\frac23$ only on the interval $\big[ \frac65, \frac32\big]$ and it is strictly less than $\frac23$ on the remainder of the parameter space.
\end{abstract}

\section{Introduction}
Binary expansions, and more generally digital expansions with an integer base and a fixed set of digits, are used in a variety of applications. For a fixed base $N \in \mathbb N_{\ge 2}$ and a finite set $A \subseteq \mathbb Z$ of digits, expressions of the form
\[ x = \sum_{k \ge 1} \frac{b_k}{N^k}, \quad b_k \in A,\]
are called {\em expansions in base $N$ with digits in $A$}. The amount of different representations that numbers have in a certain base depends on the choice of $A$. For example, when using base 2 with digit set $\{0,1\}$ all numbers in the unit interval have a unique binary expansion except for the dyadic rationals, which have precisely two. On the other hand, most numbers have uncountably many different representations in base 2 if we allow digits from $\{-1,0,1\}$. This last property can be useful for technological applications, since it provides the freedom to choose an expansion from a collection of different representations of the same number. For example, from the perspective of public key cryptography using elliptic curves there is an interest in expansions of integers having the lowest number of non-zero digits (see for example \cite{MO90,KT93,CMO98}). In case the digit set is $\{-1,0,1\}$ this is equivalent to having the lowest possible sum $\sum_{k \ge 1} |b_k|$ or lowest {\em Hamming weight}. The advantages of using digit set $\{-1,0, 1\}$ instead of $\{0,1\}$ for binary representations have been well known since the work of Reitwiesner (\cite{Rei60}) and the so called minimal weight expansions have been well studied in literature, since they lead to faster computations in the encoding process (see for example \cite{LK97,GH06,HP06,TV15} and the references therein). 

\vskip .2cm
In this article we propose a dynamical viewpoint to study minimal weight signed binary expansions. We introduce a family of maps $\{ S_\eta:[-1,1] \to [-1,1] \}_{\eta \in [1,2]}$ by setting for each $x \in [-1,1]$ the digit
\[ d_1(x) =
\begin{cases}
-1, & \text{if } x < -\frac12,\\
0, & \text{if } x \in \big[ -\frac12, \frac12\big],\\
1, & \text{if } x > -\frac12,
\end{cases}\]
and defining for each $\eta \in [1,2]$,
\begin{equation}\label{q:seta}
S_\eta(x) = 2x - d_1(x)\eta.
\end{equation}
If we now define for each $n \ge 1$ the digit $d_{\eta,n}(x) = d_1(S_\eta^{n-1}(x))$, then we can write
\begin{equation}\label{q:setan}
S_\eta^n(x) = 2S_\eta^{n-1}(x) - d_{\eta,n}(x)\eta
\end{equation}
and obtain a {\em signed binary expansion} of $x$ by
\[ x = \sum_{n \ge 1} \frac{d_{\eta, n}(x)\eta}{2^n} = \eta \sum_{n \ge 1} \frac{d_{\eta,n}(x)}{2^n},\]
or equivalently by the sequence $(d_{\eta,n}(x))_{n \ge 1}$. We call the maps $S_\eta$ {\em symmetric doubling maps}. Figure~\ref{f:variousalpha} shows the graph of $S_{\eta}$ for various $\eta$'s.
\begin{figure}[h]
\centering
\subfigure[$\eta =2$]{
\begin{tikzpicture}[scale=.9]
\draw(-1,-1)node[below]{\small -$1$}--(-.5,-1)node[below]{\small -$\frac12$}--(0,-1)node[below]{\small$0$}--(.5,-1)node[below]{\small $\frac12$}--(1,-1)node[below]{\small 1}--(1,.27)node[right]{\small \color{white}$2-\eta$}--(1,1)--(-1,1)node[left]{\small $1$}--(-1,-.27)node[left]{\small \color{white}$\eta-2$}--(-1,0)node[left]{\small $0$}--(-1,-1)node[left]{\small -1};
\draw[thick, purple!50!black](-1,0)--(-.5,1) (-.5,-1)--(.5,1)(.5,-1)--(1,0);
\draw[dotted](-1,-1)--(1,1)(-.5,-1)--(-.5,1)(.5,-1)--(.5,1);
\draw(-1,0)--(1,0)(0,-1)--(0,1);
\end{tikzpicture}}
\quad
\subfigure[$\eta =\sqrt 3$]{\begin{tikzpicture}[scale=.9]
\draw(-1,-1)node[below]{\small -$1$}--(0,-1)node[below]{\small$0$}--(.5,-1)node[below]{\small \color{white}$\frac12$}--(1,-1)node[below]{\small 1}--(1,-.73)node[right]{\small $1-\eta$}--(1,.27)node[right]{\small $2-\eta$}--(1,1)--(-1,1)node[left]{\small \color{white} $1$}--(-1,.73)node[left]{\small $\eta -1$}--(-1,0)node[left]{\small \color{white} $0$}--(-1,-.27)node[left]{\small $\eta-2$}--(-1,-1);
\draw[thick, purple!50!black](-1,-.27)--(-.5,.73) (-.5,-1)--(.5,1)(.5,-.73)--(1,.27);
\draw[dotted](-1,-1)--(1,1)(-.5,-1)--(-.5,1)(.5,-1)--(.5,1)(.5,-.73)--(1,-.73)(-1,.73)--(-.5,.73);
\draw(-1,0)--(1,0)(0,-1)--(0,1);
\draw[dotted, red, thick](1,.27)--(.27,.27)--(.27,0)(.5,-.73)--(-.73,-.73)--(-.73,.27)--(.27,.27);
\end{tikzpicture}}
\quad
\subfigure[$\eta =\frac32$]{\begin{tikzpicture}[scale=.9]
\draw(-1,-1)node[below]{\small -$1$}--(0,-1)node[below]{\small$0$}--(.5,-1)node[below]{\small \color{white}$\frac12$}--(1,-1)node[below]{\small 1}--(1,.5)node[right]{\small $\frac12$}--(1,1)--(-1,1)node[left]{\small \color{white}$1$}--(-1,-.5)node[left]{\small -$\frac12$}--(-1,-1);
\draw[thick, purple!50!black](-1,-.5)--(-.5,.5) (-.5,-1)--(.5,1)(.5,-.5)--(1,.5);
\draw[dotted](-1,-1)--(1,1)(-.5,-1)--(-.5,1)(.5,-1)--(.5,1);
\draw[dotted, red, thick](-.5,.5)--(1,.5)(-1,-.5)--(.5,-.5);
\draw(-1,0)--(1,0)(0,-1)--(0,1);
\end{tikzpicture}}
\quad
\subfigure[$\eta =\frac43$]{\begin{tikzpicture}[scale=.9]
\draw(-1,-1)node[below]{\small -$1$}--(0,-1)node[below]{\small$0$}--(.5,-1)node[below]{\small \color{white}$\frac12$}--(1,-1)node[below]{\small 1}--(1,-.33)node[right]{\small -$\frac13$}--(1,.5)node[right]{\small \color{white}$2-\eta$}--(1,.67)node[right]{\small $\frac23$}--(1,1)--(-1,1)node[left]{\small \color{white}$1$}--(-1,.33)node[left]{\small $\frac13$}--(-1,0)node[left]{\small \color{white}$\eta-2$}--(-1,-.67)node[left]{\small -$\frac23$}--(-1,-1);
\draw[thick, purple!50!black](-1,-.67)--(-.5,.33) (-.5,-1)--(.5,1)(.5,-.33)--(1,.67);
\draw[dotted](-1,-1)--(1,1)(-.5,-1)--(-.5,1)(.5,-1)--(.5,1)(.5,-.33)--(1,-.33)(-1,.33)--(-.5,.33);
\draw(-1,0)--(1,0)(0,-1)--(0,1);
\draw[dotted, red, thick](1,.67)--(.67,.67)--(.67,0)(.5,-.33)--(-.33,-.33)--(-.33,-.67)--(-.67,-.67)--(-.67,0);
\end{tikzpicture}}
\quad
\subfigure[$\eta =\frac65$]{\begin{tikzpicture}[scale=.9]
\draw(-1,-1)node[below]{\small -$1$}--(0,-1)node[below]{\small$0$}--(.5,-1)node[below]{\small \color{white}$\frac12$}--(1,-1)node[below]{\small 1}--(1,-.2)node[right]{\small -$\frac15$}--(1,.2)node[right]{\small \color{white}$2-\eta$}--(1,.8)node[right]{\small $\frac45$}--(1,1)--(-1,1)node[left]{\small \color{white}$1$}--(-1,.2)node[left]{\small $\frac15$}--(-1,-.2)node[left]{\small \color{white} $\eta-2$}--(-1,-.8)node[left]{\small -$\frac45$}--(-1,-1);
\draw[thick, purple!50!black](-1,-.8)--(-.5,.2) (-.5,-1)--(.5,1)(.5,-.2)--(1,.8);
\draw[dotted](-1,-1)--(1,1)(-.5,-1)--(-.5,1)(.5,-1)--(.5,1)(.5,-.2)--(1,-.2)(-1,.2)--(-.5,.2);
\draw(-1,0)--(1,0)(0,-1)--(0,1);
\draw[dotted, red, thick](1,.8)--(.8,.8)--(.8,.4)--(.4,.4)--(.4,.8)--(.8,.8)(.5,-.2)--(-.2,-.2)--(-.2,-.4)--(-.4,-.4)--(-.4,-.8)--(-.8,-.8)--(-.8,-.4)--(-.4,-.4);
\end{tikzpicture}}
\quad
\subfigure[$\eta=1$]{\begin{tikzpicture}[scale=.9]
\draw(-1,-1)node[below]{\small -$1$}--(0,-1)node[below]{\small$0$}--(.5,-1)node[below]{\small \color{white}$\frac12$}--(1,-1)node[below]{\small 1}--(1,1)--(-1,1)node[left]{\small \color{white}$1$}--(-1,0)node[left]{\small \color{white}-$\frac12$}--(-1,-1);
\draw[thick, purple!50!black](-1,-1)--(-.5,0) (-.5,-1)--(.5,1)(.5,0)--(1,1);
\draw[dotted](-1,-1)--(1,1)(-.5,-1)--(-.5,1)(.5,-1)--(.5,1);
\draw(-1,0)--(1,0)(0,-1)--(0,1);
\end{tikzpicture}}
\caption{The symmetric doubling map for various values of $\eta$. The red dotted lines indicate the orbits of 1 and $1-\eta$.}
\label{f:variousalpha}
\end{figure}
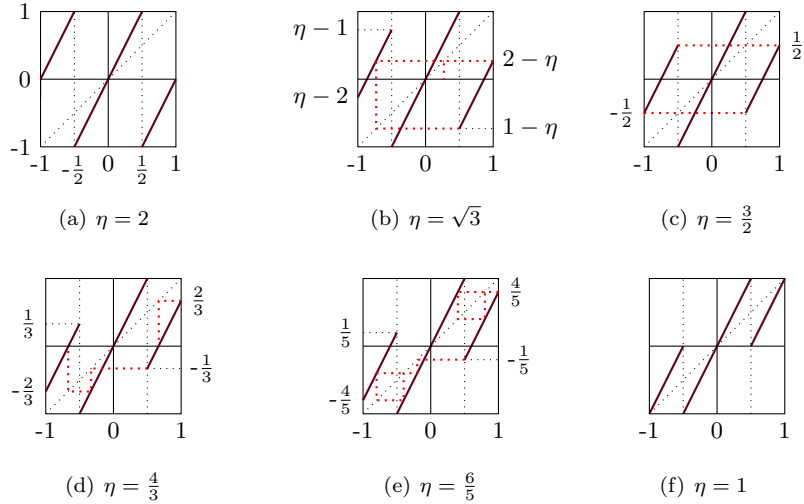

\vskip .2cm
The digit 0 occurs in position $n$ precisely when $S_\eta^{n-1}(x) \in \big[ -\frac12, \frac12\big]$. It follows easily from the literature, see for example \cite{Kop90}, that each map $S_\eta$ has a unique ergodic invariant measure $\mu_\eta$ absolutely continuous with respect to Lebesgue measure. Then by Birkhoff's Ergodic Theorem the frequency of the digit 0 in almost all of the corresponding signed binary expansions equals the $\mu_\eta$-measure of the interval $\big[-\frac12, \frac12 \big]$. The goal of this article is to find a good expression for the invariant probability density $f_\eta$ of $\mu_\eta$, so that we can calculate $\mu_\eta\big( \big[ -\frac12, \frac12\big] \big)$ and study its dependence on $\eta$.

\vskip .2cm
Results from \cite{Kop90} imply that in general $f_\eta$ is an infinite sum of indicator functions, but that $f_\eta$ becomes piecewise smooth if the map $S_\eta$ exhibits the dynamical phenomenon of matching. Matching for interval maps has received quite a lot of attention recently, especially in the case of continued fraction transformations, see \cite{NN08,DKS09,CMPT10,CT12,KSS12,BCIT13,BORG13,CT13,CM,BCK,BCMP} for example. It is the property that for each critical point the orbits of the left and right limits meet after a finite number of iterations and that the derivatives of both orbits are also equal at that time. Due to symmetry and the constant slope, for our family of maps $\{ S_\eta : [-1, 1] \to [-1,1]\}_{\eta \in [1,2]}$ we say that $S_\eta$ {\em has matching at time $m$} if $m \ge 1$ is the minimal number of iterations, such that
\[ S_\eta^m(1) = S_\eta^{m+1} \Big( \frac12^- \Big) = S_\eta^{m+1} \Big( \frac12^+ \Big) = S_\eta^m(1-\eta).\]
The power $m$ is called the {\em matching index} of $S_\eta$. In this article we give a complete description of the matching behaviour of the family $\{ S_\eta:[-1,1] \to [-1,1] \}_{\eta \in [1,2]}$, i.e., we prove the following.
\begin{main}\label{t:mintervals}
Up to a set of zero Lebesgue measure and full Hausdorff dimension, the parameter space $[1,2]$ is divided into intervals of parameters $\eta$ on which the density of the absolutely continuous invariant measure $\mu_\eta$ is piecewise smooth with the same number of jumps. These intervals are uniquely determined by the initial parts of the sequences $(d_{\eta,n}(1))_{n \ge 1}$ for $\eta$ in the interval. 
\end{main}

The intervals from Theorem~\ref{t:mintervals} are called {\em matching intervals}. If the matching index on a matching interval is $m$, then for any $\eta, \eta'$ in it we have $d_{\eta,n}(1)= d_{\eta',n}(1)$ for all $1 \le n \le m$. By close investigation of the density function of the measure $\mu_\eta$, we get the following result.
\begin{main}\label{t:23}
The map $\eta \mapsto \mu_\eta \big( \big[ -\frac12, \frac12 \big] \big)$ is continuous on $[1,2]$ and monotone on each matching interval. More precisely, if $J \subseteq [1,2]$ is a matching interval on which the matching index is $m$, then for all $\eta \in J$
\[ \mu_\eta \Big( \Big[ -\frac12, \frac12 \Big] \Big) = \frac{2^{m-1}}{2^m-1} \Big( \frac{c}{\eta} +K \Big),\]
where $c$ and $K$ are explicitly given constants depending only on $d_{\eta,n}(1)$, $1 \le n \le m-1$. 
\end{main}

On the maximal value of $ \mu_\eta \big( \big[ -\frac12, \frac12 \big] \big)$ and the frequency of the digit 0 in the signed binary expansions we have the following result.
\begin{main}\label{t:max}
It holds that $\mu_\eta \big( \big[ -\frac12, \frac12 \big] \big) \le \frac23 $ for any $\eta \in [1,2]$ and
$\mu_\eta \big( \big[ -\frac12, \frac12 \big] \big) = \frac23$ if and only if $\eta \in \big[ \frac65, \frac32 \big]$. Equivalently, for any $\eta \in [1,2]$ the frequency of the digit 0 in a typical signed binary expansion obtained from the map $S_\eta$ is at most $\frac23$ and this maximum is obtained for typical expansions if and only if $\eta \in \big[ \frac65, \frac32 \big]$. 
\end{main}



This theorem extends the results from \cite{DKL06}, where the authors focused on the signed binary expansions produced by the specific map $S_{\frac32}$ and obtained that $\mu_{\frac32} \big( \big[-\frac12, \frac12 \big] \big)=\frac23$. It is also the analogue in the context of symmetric doubling maps of the first open question mentioned in \cite{KSS12} in relation to the maximality of the entropy of Nakada's $\alpha$-continued fraction maps. To our knowledge this open question is still unanswered for the $\alpha$-continued fractions, but Theorem~\ref{t:max} gives a positive answer for the piecewise linear analogue.

\vskip .2cm
An ingredient in the proof of these results is a correspondence between the maps $S_\eta$, the doubling map, the tent map, the Gauss map and Nakada's $\alpha$-continued fraction maps. As an immediate consequence of this correspondence we obtain a lot of information on the subset $\mathcal N \subseteq [1,2]$ of those parameters $\eta$ for which $S_\eta$ does not have matching. In particular we will see that, besides being a Lebesgue null set of full Hausdorff dimension, $\mathcal N$ is also totally disconnected and closed and that to each matching interval for the family $\{ S_\eta\}_{\eta \in [1,2]}$ there corresponds an element in the set $\mathcal N$ that is transcendental. Moreover, the set $\mathcal N$ is related to interesting subsets of the parameter space of several other one-parameter families of dynamical systems as explained in Remark~\ref{r:othersets}.

\vskip .2cm
The article is organised as follows. In the second section we introduce notation regarding digit sequences and provide the necessary background on the doubling, tent and Gauss map. We then prove that $S_\eta$ has matching for Lebesgue almost all $\eta \in [1,2]$ and link the non-matching set $\mathcal N$ to another set related to the tent map. In the third section we describe the connection between the symmetric doubling maps and the $\alpha$-continued fraction maps and prove Theorem~\ref{t:mintervals}. 
In the fourth section we closely examine the density function $f_\eta$ on the matching intervals and prove Theorem~\ref{t:23}. In the last section we prove Theorem~\ref{t:max}.

\section{First results on matching and the frequency of 0}

\subsection{Preliminaries on sequences and dynamical systems}\label{ss:dynsys}
We will often switch between the dynamics of the systems on the interval and their symbolic codings. Therefore, we first introduce some notation regarding sequences of digits. For a finite or countable alphabet $\mathcal A \subseteq \mathbb Z$ let $\mathcal A^*$ denote the set of all finite strings of symbols from $\mathcal A$, called {\em words}, and let $\varepsilon$ denote the empty word. For any word $w_1 \cdots w_m \in \mathcal A^*$ and any $1 \le i < j \le m$ write $w_i^j := w_i \cdots w_j$. The set of one-sided infinite sequences of symbols from $\mathcal A$ is denoted by $\mathcal A^\mathbb N$ and on this set the {\em left shift}, denoted by $\sigma :\mathcal A^\mathbb N \to \mathcal A^\mathbb N$, is given by $\sigma (b_n)_{n \ge 1} = (b_{n+1})_{n \ge 1}$. On words and sequences we consider the lexicographical ordering denoted by $\prec$. Since the orbit of 1 under $S_\eta$ is of particular importance, we denote the corresponding digit sequence by
\begin{equation}\label{q:d1}
d_\eta := (d_{\eta,n}(1))_{n \ge 1}.
\end{equation}
The next lemma on signed binary expansions will be of use later.
\begin{lemma}\label{l:order}
Let $\eta_1,\eta_2\in (1,2)$, then $\eta_1 < \eta_2$ if and only if $d_{\eta_2} \prec d_{\eta_1}$.
\end{lemma}

\begin{proof}
First assume that $\eta_1<\eta_2$. From the definition of the sequences $d_\eta$ in \eqref{q:d1} it follows that,
\[ \eta_1 \sum_{n \ge 1} \frac{d_{\eta_1,n}}{2^n} =1 = \eta_2 \sum_{n \ge 1} \frac{d_{\eta_2,n}}{2^n}.\]
Hence, $d_{\eta_1} \neq d_{\eta_2}$. Let $n$ be the smallest index such that $d_{\eta_2,n+1} \neq d_{\eta_1,n+1}$. From \eqref{q:setan} it follows that the orbits of $S_\eta$ are determined as follows:
\begin{equation}\label{q:Torbit}
S_{\eta}^n (x) = 2^nx - d_{\eta,1}(x) 2^{n-1}\eta - \cdots - d_{\eta,n-1}(x) 2\eta - d_{\eta,n}(x) \eta.
\end{equation}
Write
\[ x_n = d_{\eta_2,1} 2^{n-1}+\cdots +d_{\eta_2,n-1}2 +d_{\eta_2,n} = d_{\eta_1,1} 2^{n-1}+\cdots +d_{\eta_1,n-1}2 + d_{\eta_1,n}.\]
Then by (\ref{q:Torbit}),
 \[ S^n_{\eta_2}(1) =2^n-x_n\eta_2 < 2^n - x_n \eta_1 = S^n_{\eta_1}(1).\]
This gives $d_{\eta_2,n+1} < d_{\eta_1,n+1}$.

\medskip
\noindent
Now assume $d_{\eta_2} \prec d_{\eta_1}$, and let $n$ be the first index such that $d_{\eta_2, n+1}<d_{\eta_1, n+1}$. Then by (\ref{q:Torbit}),
\[  2^n-\Big(d_{\eta_2,1} 2^{n-1}+\cdots +d_{\eta_2,n}\Big)\eta_2 = S^n_{\eta_2}(1)  < S^n_{\eta_1}(1)  =2^n-\Big(d_{\eta_1,1} 2^{n-1}+\cdots +d_{\eta_1,n}\Big)\eta_1,\]
implying $\eta_2>\eta_1$.
\end{proof}

Define the binary {\em valuation function} $v:\{-1,0,1\}^\mathbb N \to \mathbb R$ by
\begin{equation}\label{q:decdot}
v((b_n)_{n \ge 1}) := \sum_{n \ge 1} \frac{b_n}{2^n}.
\end{equation}
For any word $w \in \{0,1\}^*$ we put $v(w):= v(w0^\infty)$. The proofs of a number of our results are based on the relation the symmetric doubling maps $\{ S_\eta\}_{\eta \in [1,2]}$ have to some other well known interval maps. We gather some necessary basic information on these other systems here.

\vskip .2cm
The {\em doubling map} $D$ is given by
\[ D: [0,1] \to [0,1], \, x \mapsto 2x \pmod 1,\]
see Figure~\ref{f:four}(a). Lebesgue measure is invariant and ergodic for $D$. Furthermore, $D$ can be used to generate {\em binary expansions} of numbers in $[0,1]$ as follows. For each $n \ge 1$, set $b_n(x) = 0$ if $D^{n-1}(x) < \frac12$ and $b_n(x)=1$ if $D^{n-1}(x) \ge \frac12$. Then one can write $D^n(x)=2D^{n-1}(x)-b_n(x)$ and so $x = v((b_n)_{n \ge 1}) = \sum_{n \ge 1} \frac{b_n(x)}{2^n}$. The doubling map preserves the lexicographical ordering on the binary expansions: for $x,y \in [0,1]$ we have
\[ x < y \Leftrightarrow (b_n(x))_{n \ge 1} \prec (b_n(y))_{n \ge 1}.\]
The doubling map is closely related to the {\em tent map}, shown in Figure~\ref{f:four}(b) and given by
\begin{equation}\label{q:T}
T:[0,1] \to [0,1], \, x \mapsto \min\{ 2x,2-2x\}.
\end{equation}
The orbits of the doubling map and of the tent map are related via
\[T \circ D = T \circ T.\]

\vskip .2cm
Besides binary expansions, another famous way to represent real numbers is with a continued fraction. There are many different types of continued fractions. The most common ones are called regular continued fractions and they are generated by iterations of the {\em Gauss map} $G:[0,1] \to [0,1]$ given by $G(0)=0$ and
\[G(x) = \frac1x \pmod 1, \quad \text{if } x \neq 0,\]
see Figure~\ref{f:four}(c). If we set for each $n \ge 1$ and each $x \in [0,1]$ with $G^{n-1}(x) \neq 0$ that $a_n (x) =k$ if $G^{n-1}(x) \in \big( \frac1{k+1}, \frac1k \big]$, then it is well known that if $x \not \in \mathbb Q$, then
\[ x = \cfrac1{a_1(x) + \cfrac1{a_2(x) + \cfrac1{a_3(x) + \ddots}}}\]
and if $x \in \mathbb Q$, then there is an $m \in \mathbb N$, such that
\[ x = \cfrac1{a_1(x) + \cfrac1{a_2(x) + \ddots + \cfrac1{a_m(x)}}}.\]
This expression is called the {\em regular continued fraction expansion} of $x$ and is characterised by the fact that all numerators in the continued fraction are equal to 1. For a sequence of numbers $(a_n)_{n \ge 1} \subseteq \mathbb N^{\mathbb N}$ we use the notation
\[ [0; a_1 a_2 \cdots ] := \cfrac1{a_1+\cfrac1{a_2 + \ddots }}. \]
Note that in case of a finite number of digits $a_1, \ldots, a_m$ with $a_m \ge 2$, we have
\[  [0;a_1  \cdots a_m] = [0; a_1 \cdots  a_{m-1}  (a_m-1)1]. \]
If for an $x \in (0,1]$ there is an infinite sequence of digits $(a_n)_{n \ge 1} \subseteq \mathbb N^{\mathbb N}$, such that $x=[0;a_1a_2\cdots]$, then $(a_n)_{n \ge 1}$ is the digit sequence produced by the Gauss map. On the other hand, if $x$ is represented by two finite sequences
\[ x= [0; a_1 a_2 \cdots a_m ] := [0;a_1a_2 \cdots a_{m-1}(a_m-1)1], \]
then the digit sequence for $x$ given by the Gauss map is $(a_n)_{1 \le n \le m}$. Also note that for each $x \in (0,1]$ we have
\[ G(x) = [0; a_2(x)a_3(x)\cdots], \]
i.e., $(a_n(G(x)))_{n \ge 1} = \sigma (a_n(x))_{n \ge 1}$. Let $<_g$ denote the ordering on $\mathbb N^{\mathbb N}$ induced by the Gauss map on the regular continued fraction expansions, i.e., $(c_n)_{n \ge 1} <_g (d_n)_{n \ge 1}$ if and only if the first index $n$ such that $c_n \neq d_n$ is even and $c_n < d_n$ or it is odd and $c_n > d_n$. Then for $x=[0; c_1  c_2  \cdots ]$ and $y=[0; d_1 d_2 \cdots ]$ we have
\begin{equation}\label{q:gaussord}
x < y \quad  \Leftrightarrow \quad (c_n)_{n \ge 1} <_g (d_n)_{n\ge 1}.
\end{equation}
The above ordering is also applied to finite strings of the same length.

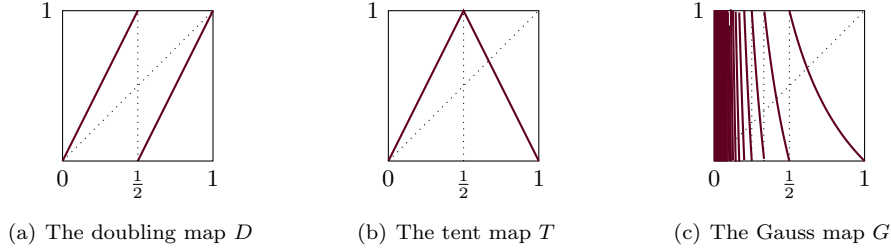
\begin{figure}[h]
\centering
\subfigure[The doubling map $D$]{\begin{tikzpicture}[scale=2]
\draw[white](-.6,0)--(1.5,0);
\draw(0,0)node[below]{\small $0$}--(.5,0)node[below]{\small $\frac12$}--(1,0)node[below]{\small$1$}--(1,1)--(0,1)node[left]{\small $1$}--(0,0);
\draw[thick, purple!50!black](0,0)--(.5,1)(.5,0)--(1,1);
\draw[dotted](0,0)--(1,1)(.5,0)--(.5,1);
\end{tikzpicture}}
\subfigure[The tent map $T$]{\begin{tikzpicture}[scale=2]
\draw[white](-.6,0)--(1.5,0);
\draw(0,0)node[below]{\small $0$}--(.5,0)node[below]{\small $\frac12$}--(1,0)node[below]{\small$1$}--(1,1)--(0,1)node[left]{\small $1$}--(0,0);
\draw[thick, purple!50!black](0,0)--(.5,1)--(1,0);
\draw[dotted](0,0)--(1,1)(.5,0)--(.5,1);
\end{tikzpicture}}
\subfigure[The Gauss map $G$]{\begin{tikzpicture}[scale=2]
\draw[white] (-.6,0)--(1.5,0);
\draw(0,0)node[below]{\small $0$}--(.5,0)node[below]{\small $\frac12$}--(1,0)node[below]{\small$1$}--(1,1)--(0,1)node[left]{\small $1$}--(0,0);
 \draw[dotted] (.5,0)--(.5,1);
  \draw[dotted] (.33,0)--(.33,1);
  \draw[dotted] (.25,0)--(.25,1);
 \draw[dotted] (0,0)--(1,1);
\draw[thick, purple!50!black, smooth, samples =20, domain=.5:1] plot(\x,{1 / \x -1});
\draw[thick, purple!50!black, smooth, samples =20, domain=.334:.5] plot(\x,{1 /\x -2});
\draw[thick, purple!50!black, smooth, samples =20, domain=.25:.332] plot(\x,{1 /\x -3});
\draw[thick, purple!50!black, smooth, samples =20, domain=.2:.25] plot(\x,{1 /\x -4});
\draw[thick, purple!50!black, smooth, samples =20, domain=.167:.2] plot(\x,{1 /\x -5});
\draw[thick, purple!50!black, smooth, samples =20, domain=.143:.167] plot(\x,{1 /\x -6});
\draw[thick, purple!50!black, smooth, samples =20, domain=.125:.143] plot(\x,{1 /\x -7});
\draw[thick, purple!50!black, smooth, samples =20, domain=.111:.125] plot(\x,{1 /\x -8});
\draw[thick, purple!50!black, smooth, samples =20, domain=.101:.1112] plot(\x,{1 /\x -9});
\draw[thick, purple!50!black, smooth, samples =20, domain=.091:.1] plot(\x,{1 /\x -10});
\filldraw[purple!50!black] (0,0) rectangle (.09,1);
\end{tikzpicture}}
\caption{The maps $D$, $T$ and $G$.}
\label{f:four}
\end{figure}

\subsection{The invariant density and the definition of matching}
The frequency of the digit 0 in the signed binary expansions produced by the symmetric doubling map $S_\eta$, $\eta \in [1,2]$, from (\ref{q:seta}) can be obtained from Birkhoff's Ergodic Theorem if we would have an explicit expression for an absolutely continuous invariant measure. For any $\eta \in [1,2]$ the map $S_\eta$ falls into the class of maps studied by Kopf in \cite{Kop90}. The results from \cite{Kop90} imply that for any $\eta \in (1,2]$ the system $S_\eta$ has a unique absolutely continuous invariant measure $\mu_{\eta}$ that is ergodic and has a density function $f_\eta$ that is an infinite sum of indicator functions over intervals with endpoints in the set
\begin{equation}\label{q:partition}
\{ S^n_{\eta} (1-\eta), S_{\eta}^n (1), S^n_{\eta} (\eta-1), S_{\eta}^n (-1) \, : \, n \ge 0 \}.
\end{equation}
More specifically, in Appendix A we show that the probability density $f_\eta$ is given by
\begin{equation}\label{q:density}
f_\eta(x) = \frac1{C} \sum_{n \ge 0} \frac{1}{2^{n+1}} \Big(1_{[-1, S_\eta^n (\eta-1))}(x) - 1_{[-1, S_\eta^n (-1))}(x) + 1_{[-1, S_\eta^n (1))}(x) - 1_{[-1, S_\eta^n (1-\eta))}(x) \Big),
\end{equation}
where $C$ is a normalising constant. From \eqref{q:density} we see that there are two situations in which the infinite sum becomes a finite sum and the density becomes piecewise smooth. Firstly, this happens when the orbits of 1 and $1-\eta$ under $S_\eta$ are finite (and thus by symmetry also the orbits of $-1$ and $\eta-1$ are finite). In that case the set from (\ref{q:partition}) becomes finite and $S_\eta$ has a Markov partition. For a concrete example, consider $\eta=\frac32$ from Figure~\ref{f:variousalpha}(c). A Markov partition is given here by
\[ \Big\{ \Big(-1,-\frac12 \Big), \Big( -\frac12, \frac12 \Big), \Big( \frac12, 1\Big) \Big\}.\]

\vskip .2cm
From the description of $S_\eta^n$ from \eqref{q:Torbit} it is clear that if the orbit of 1 is finite, i.e., if there are $k \neq n$ such that $S^n_\eta(1) = S^k_\eta(1)$, then $\eta \in \mathbb Q$. Hence, for most $\eta$ a Markov partition does not exist. From \eqref{q:density} we see that $f_\eta$ is also piecewise smooth if there is an $m \ge 1$, such that $S_{\eta}^m (1) = S_{\eta}^m (1-\eta)$, i.e., if $S_\eta$ has matching. Recall the definition of matching from the introduction:

\begin{defn}\label{d:matching}
The map $S_\eta$ has {\em matching} if there is an $m \ge 1$, such that $S^m_\eta (1) = S^m_\eta (1-\eta)$. For $S_\eta$ we define the {\em matching index} $m(\eta)$ by
\[ m(\eta):= \inf \{m\ge 1 \, : \, S^m_\eta (1) = S^m_\eta (1-\eta)\}.\]
If $S_\eta$ does not have matching, then $m(\eta)=\infty$. 
\end{defn}

\begin{remark}{\rm
Matching does not exclude a Markov partition and vice versa. For $\eta=\frac65$ for example (see Figure~\ref{f:variousalpha}(e)), we have a Markov partition, but we do not have matching. The orbits of $1$ and $1-\eta$ are given by:
\begin{align*}
&S_\eta(1)= 2-\eta = \frac45, && S^2_\eta(1) = \frac25, & & S^3_\eta(1)=S_\eta(1),\\
&S_\eta(1-\eta) = S_\eta\Big( -\frac15 \Big) = -\frac25, & & S_\eta^2(1-\eta) = -\frac45, && S_\eta^3(1-\eta) =S_\eta(1-\eta). 
\end{align*}
Note that $d_\eta = 1(10)^\infty$. For $\eta=\frac43$ (see Figure~\ref{f:variousalpha}(d)) there is matching and a Markov partition since $S_\eta^2 (1) = 0 = S^2_\eta( 1-\eta)$. For $\eta = \sqrt 3$ (see Figure~\ref{f:variousalpha}(b)) we have matching, but no Markov partition. The orbits of 1 and $1-\eta$ are given by 
\[ S_\eta (1) = 2-\sqrt3 \quad \text{and} \quad S_\eta(1-\sqrt 3) = 2-2\sqrt3 + \sqrt 3  = S_\eta (1).\]
}\end{remark}

\vskip .2cm
There are two cases, namely $\eta=1$ and $\eta \in \big[\frac32, 2\big]$, for which we can immediately determine whether matching occurs and compute the frequency of 0 in the sequence $(d_{\eta,n}(x))_{n \ge 1}$ for typical $x$. If $\eta =1$, then $S^n_{\eta} (1)=1=\eta$ for all $n\ge 0$ and $S^n_{\eta} (1-\eta) = 0$ for all $n \ge 0$. So  $S^n_{\eta} (1) - S^n_{\eta} (1-\eta) = 1 = \eta$ for all $n \ge 0$ and there is no matching. In this case the system splits into two copies of the doubling map, as can be seen from Figure~\ref{f:variousalpha}(f). Normalised Lebesgue measure is invariant and $\mu_1 \big( \big[-\frac12, \frac12 \big] \big)=\frac12$, so typically half of the digits in $(d_{1,n}(x))_{n \ge 1}$ will be a 0. For $\eta = \frac32$ we know from Figure~\ref{f:variousalpha}(c) that $S_\eta$ has a Markov partition. If $\eta \in \big(\frac32, 2\big]$, then $1-\eta < -\frac12$ and $\eta -1 > 2-\eta$. This gives that $S_\eta(1-\eta) = 2(1-\eta)+\eta = 2-\eta = S_\eta (1)$. Hence, we have matching after one step and we have identified our first matching interval. By \eqref{q:density} we get that for $\eta \in \big[\frac32, 2\big]$ the invariant probability density $f_{\eta}$ is given by
\begin{equation}\label{q:h2}
f_{\eta}:[-1,1]\to[-1,1], \, x \mapsto \frac1{2\eta} \Big(1 + 1_{(1-\eta , \eta-1)}(x)\Big).
\end{equation}
So,
\begin{equation}\label{q:>32}
\mu_{\eta}\Big( \Big[-\frac12, \frac12 \Big] \Big) = \frac1{2\eta} \int_{-\frac12}^{\frac12} 1+ 1_{(1-\eta, \eta-1)}(x) \, dx = \frac1{\eta},
\end{equation}
which on the interval $\big[\frac32, 2\big]$ is maximal for $\eta =\frac32$, giving $\mu_{\frac32}\big( \big[-\frac12, \frac12 \big] \big) = \frac23$. Hence, the maximal frequency of the digit 0 in the sequences $(d_{\eta,n}(x))_{n\ge 1}$ for typical $x$ and $\eta \in \big[ \frac32,2\big]$ is obtained for $\eta = \frac32$ and equals $\frac23$.

\subsection{Matching almost everywhere}
From now on we let $\eta \in \big(1, \frac32 \big)$. By Definition~\ref{d:matching} we have matching at time $m$ if $S^m_{\eta}(1) = S^m_{\eta} (1-\eta)$ and $S^n_{\eta} (1) \neq S^n_{\eta} (1-\eta)$ for all $1 \le n \le m-1$. The next result says that under iterations of $S_\eta$ the points $1$ and $1-\eta$ are either a distance $\eta$ apart or mapped to the same point.
\begin{prop}\label{p:0alpha} For any $n \ge 0$, $S^n_{\eta}(1) - S^n_{\eta} (1-\eta) \in \{0, \eta \}$. Moreover, $m(\eta)=m$ if and only if $S^{m-1}_\eta(1)> \frac12$ and $S^{m-1}_\eta(1-\eta)< -\frac12$.
\end{prop}

\begin{proof}
We prove the first statement by induction.\\
- For $n =0$ it is true, since $1- (1-\eta)= \eta$.\\
- Assume now that for some $n \ge 0$, $S^n_{\eta}(1) - S^n_{\eta} (1-\eta) \in \{0, \eta \}$. If $S^n_{\eta}(1) - S^n_{\eta} (1-\eta)=0$, then $S^k_{\eta}(1) - S^k_{\eta} (1-\eta) =0$ for all $k \ge n$, so assume that $S^n_{\eta}(1) - S^n_{\eta} (1-\eta) = \eta$. From (\ref{q:Torbit}) it follows that there are $b,c \in \mathbb Z$ such that
\[ S^n_{\eta}(1) = 2^n - b\eta \quad \text{ and } \quad S^n_{\eta}(1-\eta) = 2^n - c\eta.\]
This implies that $-b\eta + c\eta = \eta$, so $c=1+b$. Since $\eta > 1$, $S^n_{\eta}(1) > 0$ and $S^n_{\eta}(1-\eta) <0$. Moreover, $S^n_{\eta}(1)$ and $S^n_{\eta}(1-\eta)$ cannot both lie in the interval $\big[ -\frac12, \frac12 \big]$. We distinguish three cases.

\vskip .1cm
\underline{Case 1:} Assume $0 < S^n_{\eta} (1) < \frac12$, so $S^n_{\eta} (1-\eta) < -\frac12$. Then $S^{n +1}_{\eta}(1) = 2^{n +1} -2b\eta$ and $S^{n+1}_{\eta} (1-\eta) = 2^{n+1}-2c\eta + \eta = 2^{n+1} - 2(b+1)\eta + \eta = 2^{n +1}-2b\eta -\eta$. Hence, $S^{n+1}_{\eta}( 1) - S^{n+1}_{\eta} (1-\eta) = \eta$.

\vskip .1cm
\underline{Case 2:} Assume $S^n_{\eta}(1) > \frac12$ and $-\frac12 < S^n_{\eta} (1-\eta) < 0$. Then $S^{n +1}_{\eta}(1) = 2^{n+1}-2b\eta -\eta$ and $S^{n+1}_{\eta} (1-\eta) = 2^{n+1}-2c\eta = 2^{n+1} - 2(b+1)\eta = 2^{n+1}-2b\eta -2\eta$. So, $S^{n+1}_{\eta} (1) - S^{n+1}_{\eta} (1-\eta) = \eta$.

\vskip .1cm
\underline{Case 3:} Assume $S^n_{\eta}(1) > \frac12$ and $ S^n_{\eta} (1-\eta) <  -\frac12$. Then $S^{n +1}_{\eta}(1) = 2^{n +1}-2b\eta -\eta$ and $S^{n+1}_{\eta} (1-\eta) = 2^{n+1}-2c\eta + \eta= 2^{n+1} - 2(b+1)\eta +\eta= 2^{n +1}-2b\eta -\eta$. So, $S^{n+1}_{\eta} (1) - S^{n+1}_{\eta} (1-\eta) = 0$ and matching occurs at step $n +1$.

\vskip .1cm
In case $S^n_{\eta} (1)= \frac12$, the map $S_\eta$ has a Markov partition, since $S^n_{\eta}(1-\eta) = \frac12 -\eta$ and $\eta >1$ imply that $S^{n+1}_{\eta} (1-\eta) = 1-2\eta + \eta =1-\eta$. A similar situation occurs when $S^n_{\eta} (1-\eta)=-\frac12$ and $S^n_{\eta} (1) = \eta -\frac12 > \frac12$. Then $S^{n +1}_{\eta}(1) = 2\eta -1-\eta = \eta -1$. By (\ref{q:seta}) $S_\eta$ does not have matching in these cases and $S^n_{\eta}(1) - S^n_{\eta} (1-\eta) = \eta$ for all $n \ge 0$. Hence we see that the first statement holds for all $n \in \mathbb N$. For the second statement note that the only way in which the difference between $S_\eta^n(1)$ and $S_\eta^n(1-\eta)$ becomes 0 is precisely in case 3 from the induction proof.
\end{proof}

\begin{remark}{\rm Consider again the situation that $S^n_{\eta} (1)= \frac12$ or $S^n_{\eta} (1-\eta)=-\frac12$ for some $n$. Whether or not matching occurs depends on the choice made for the action of $S_\eta$ on the critical points. Our choice for the definition from (\ref{q:seta}) implies that $S_\eta$ does not have matching in these cases.
}\end{remark}

From this proposition we obtain an alternative characterisation of the matching index $m(\eta)$ from Definition~\ref{d:matching}:
\[ m(\eta) = \inf \Big\{ n \ge 1 \, : \, \frac12 < S^n_{\eta} (1) < \eta- \frac12 \Big\} +1.\]

\vskip .2cm
We now establish the relation between the map $D$ and the maps $S_\eta$, that will be of help later. Recall the definition of $d_\eta = (d_{\eta,n})_{n \ge 1}= (d_{\eta,n}(1))_{n \ge 1}$ from \eqref{q:d1}.

\begin{prop}\label{relationbinary}
Let $1 < \eta < \frac32$ and let $m$ be the first index such that either $S^m_\eta(1) \in \big( \frac12, \eta - \frac12 \big)$ or $D^m\big( \frac1{\eta}\big) \in \big( \frac1{2\eta}, 1-\frac1{2\eta} \big)$. Then $\frac1{\eta} S^n_\eta(1) = D^n\big( \frac1{\eta} \big)$ and thus $d_{\eta,n} = b_n \big( \frac1{\eta} \big)$ for all $0 \le n \le m$. Moreover, both $S^m_\eta(1) \in \big( \frac12, \eta - \frac12 \big)$ and $D^m\big( \frac1{\eta}\big) \in \big( \frac1{2\eta}, 1-\frac1{2\eta} \big)$.
\end{prop}

\begin{proof}
Note that the last statement immediately follows from the fact that $\frac1{\eta} S^m_\eta(1) = D^m\big( \frac1{\eta} \big)$. We prove the first statement by induction.\\
- For $n=0$ the statement holds, since $\eta D^0\big( \frac1{\eta} \big) = 1 = S_\eta^0(1)$ and so $b_1\big( \frac1{\eta} \big)=1=d_{\eta,1}$.\\
- Suppose that for some $n < m$ we have $\frac1{\eta} S^j_\eta(1) = D^j \big( \frac1{\eta} \big)$ and $d_{\eta,j} = b_j \big( \frac1{\eta} \big)$ for all $j \le n$. Similar to (\ref{q:Torbit}) it holds that
\[ D^n \Big( \frac1{\eta} \Big) = \frac{2^n}{\eta} - b_1\Big( \frac1{\eta} \Big) 2^{n-1} - \cdots -b_n\Big( \frac1{\eta} \Big) = \frac{2^n}{\eta} - d_{\eta,1} 2^{n-1} - \cdots -d_{\eta, n}.\]
So, $\frac1{\eta} S^n_{\eta}(1) = D^n \big(\frac1{\eta} \big)$. We have $d_{\eta,n+1}=1$ if and only if $S^n_{\eta}(1) \in \big( \eta - \frac12, 1 \big]$, which by the previous holds if and only if $D^n\big( \frac1{\eta} \big) \in \big( 1-\frac1{2\eta}, \frac1{\eta} \big] \subseteq \big( \frac12,1\big)$. Hence, $d_{\eta,n+1}=1$ if and only if $b_{n+1} \big( \frac1{\eta} \big) =1$. On the other hand, $d_{\eta,n+1}=0$ if and only if $S^n_{\eta}(1) \in \big( \eta -1, \frac12 \big]$, which holds if and only if $D^n\big( \frac1{\eta} \big) \in \big( 1-\frac1{\eta}, \frac1{2\eta} \big] \subseteq \big( 0, \frac12\big)$. Hence $d_{\eta,n+1}=0$ if and only if $b_{n+1} \big( \frac1{\eta} \big) =0$. This implies that
\[ S^{n+1}_\eta (1) = 2 S^n_\eta(1) -d_{\eta, n+1}\eta = \eta \Big(2D^n\Big( \frac1{\eta} \Big) - b_{n+1} \Big( \frac1{\eta} \Big) \Big) = \eta D^{n+1}\Big( \frac1{\eta} \Big),\]
which proves the proposition.
\end{proof}

Figure~\ref{f:holes} shows the regions where matching occurs for $S_\eta$ and $D$. The previous proposition immediately implies a third characterisation of the matching index:
\[ m(\eta) = \inf\Big\{ n\ge 1\, : \, D^n\Big(\frac1{\eta}\Big) \in \Big( \frac1{2\eta}, 1-\frac1{2\eta} \Big)\Big\} +1.\]
We use this characterisation in the next proposition.
\begin{figure}[h]
\centering
\subfigure[$S_\eta$ with the hole $\big( \frac12, \eta-\frac12 \big)$]{\begin{tikzpicture}[scale=1]
\draw[white](-3,0)--(3,0);
\filldraw[yellow!35] (.5,-1) rectangle (.71,1);
\draw(-1,-1)node[left]{\small -$1$}--(-.5,-1)node[below]{\small -$\frac12$}--(0,-1)node[below]{\small$0$}--(.5,-1)node[below]{\small$\frac12$}--(1,-1)node[below]{\color{red}\small 1}--(1,1)--(-1,1)node[left]{\small $1$}--(-1,0)node[left]{\small $0$}--(-1,-1);
\draw[thick, purple!50!black](-1,-.79)--(-.5,.21) (-.5,-1)--(.5,1)(.5,-.21)--(1,.79);
\draw[dotted](-1,-1)--(1,1)(-.5,-1)--(-.5,1)(.5,-1)--(.5,1)(.71,-1)--(.71,1);
\draw(-1,0)--(1,0)(0,-1)--(0,1);
\filldraw[red] (1,.79) circle (1.6pt);
\end{tikzpicture}}
\hspace{.5cm}
\subfigure[$D$ with the hole $\big(\frac1{2\eta}, 1-\frac1{2\eta}\big)$]{\begin{tikzpicture}[scale=2]
\draw[white](-1,0)--(2,0);
\filldraw[yellow!35] (.413,0) rectangle (.587,1);
\draw(0,0)node[below]{\small $0$}--(.413,0)node[below]{\small $\frac1{2\eta}$}--(.826,0)node[below]{\color{red}\small $\frac1{\eta}$}--(1,0)node[below]{\small$1$}--(1,1)--(0,1)node[left]{\small $1$}--(0,0);
\draw[thick, purple!50!black](0,0)--(.5,1)(.5,0)--(1,1);
\draw[dotted](0,0)--(1,1)(.5,0)--(.5,1)(.413,0)--(.413,1)(.587,0)--(.587,1);
\draw[dotted, red] (.826,0)--(.826,.652);
\filldraw[red] (.826, .652) circle (.8pt);
\end{tikzpicture}}
\caption{There is matching for $S_\eta$ if the orbit of 1 enters the yellow region under $S_\eta$ in (a) or equivalently if the orbit of $\frac1\eta$ enters the yellow region under $D$ in (b).}
\label{f:holes}
\end{figure}
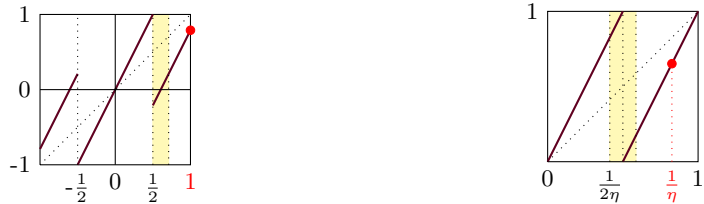

\begin{prop}\label{t:measure0}
For Lebesgue almost all $\eta \in [ 1 , 2]$ it holds that $m(\eta) < \infty$, i.e., $S_\eta$ has matching.
\end{prop}

\begin{proof}
It is enough to consider $\eta \in \big(1, \frac32 \big)$. Let $k \ge 7$. The ergodicity of $D$ with respect to Lebesgue measure gives that for Lebesgue almost every $x \in (0,1)$ there is an $n \ge 1$, such that $D^n x \in \big( \frac12-\frac1k, \frac12 + \frac1k \big)$. Since $\eta > \frac{k}{k-2}$ if and only if $\frac12 + \frac1k < 1-\frac1{2\eta}$, this means that for almost all $\eta \in \big( \frac{k}{k-2}, \frac32 \big)$ matching occurs for $S_{\eta}$. Let $A_k$ denote the set of all $\eta \in \big( \frac{k}{k-2}, \frac32 \big)$ such that $S_{\eta}$ does not have matching. Then $A_k$ has zero Lebesgue measure and thus also $\bigcup_{k \ge 7} A_k$ has zero Lebesgue measure. Since $\bigcup_{k \ge 7} A_k$ equals the set of all $\eta \in \big(1 , \frac32 \big)$ such that $S_{\eta}$ does not have matching, this finishes the proof.
\end{proof}

Now consider the {\em non-matching set} $\mathcal N$ of parameters $\eta$ such that $S_\eta$ does not have matching:
\begin{equation}
\label{q:nonmatching}
\mathcal N := \Big\{ \eta \in \Big(1, \frac32 \Big) \, : \, m(\eta) = \infty \Big\} =  \Big\{  \eta \in \Big(1, \frac32 \Big) \, :  \, D^n \Big( \frac1{\eta} \Big) \not \in \Big( \frac1{2\eta}, 1-\frac1{2\eta} \Big), \mbox{ for all } n\ge 1  \Big\}. 
\end{equation}
One easily checks that $\mathcal N = \frac1{\Gamma}$, where
\begin{equation}\label{q:gamma}
\Gamma =\{x\in [0,1] \, : \, 1-x \le D^k(x) \le x, \mbox{ for all } k\ge 1\},
\end{equation}
by noting that if $x\in \Gamma$, then $x>2/3$, and $D^k(x)\notin \big(\frac{x}{2}, 1-\frac{x}{2} \big)$ for all $k$. The set $\Gamma$ was introduced in \cite{AC83,AC01} by Allouche and Cosnard in connection with univoque numbers. Recall the definition of the tent map from \eqref{q:T}. In \cite[Lemma 5.5]{BCIT13} Bonanno et al.~proved that $\Gamma = \Lambda \backslash \{0\}$, where 
\begin{equation}\label{q:lambda}
\Lambda=\{x\in [0,1]\, : \, T^k(x)\le x, \mbox{ for all } k\ge 1\},
\end{equation}
which under an appropriate coding can be seen as a representation of all the kneading invariants of unimodal maps. They showed, among other things, that the derived set $\Lambda'$, i.e., the set $\Lambda$ minus its isolated points, is a Cantor set. In particular this yields that $\Lambda$ is closed, uncountable, totally disconnected and has Lebesgue measure 0. In \cite{CT12} it was shown that $\Lambda$ has Hausdorff dimension 1. Since $\mathcal N$ is homeomorphic to $\Lambda \backslash \{ 0 \}$ via the bi-Lipschitz homeomorphism $x \mapsto \frac1x$ on $[1,2]$, we can use these correspondences to conclude that $\mathcal N$ has the following properties:
\begin{itemize}
\item[(i)] the set $\mathcal N$ has Lebesgue measure 0,
\item[(ii)] the set ${\mathcal N}$ has Hausdorff dimension 1,
\item[(iii)] the set ${\mathcal N}$ is totally disconnected and closed.
\end{itemize}
Note that these properties also imply that matching occurs Lebesgue almost everywhere.

\section{Matching intervals for the symmetric doubling maps}
Now that we have established that matching holds on a full measure set, we will study the finer matching structure of the maps $S_\eta$. In particular, we want to identify intervals in $[1,2]$ of parameters such the maps $S_\eta$ have the same matching index for all $\eta$ in such an interval. This is achieved by establishing a relation between the symmetric doubling maps and another one-parameter family of maps, Nakada's $\alpha$-continued fraction maps. In the next section we summarise the information on $\alpha$-continued fraction maps that we need here.

\subsection{Matching for the $\alpha$-continued fraction maps}\label{ss:alpha}
The $\alpha$-continued fraction maps can be seen as generalisations of the Gauss map. They were first introduced by Nakada in \cite{Nak81} and form a one-parameter family of maps $\{T_\alpha: [\alpha-1,\alpha] \to [\alpha-1, \alpha] \}_{\alpha \in [0,1]}$ defined by $T_\alpha(0)=0$ and
\[ T_\alpha (x) = \left| \frac1x \right| - \left\lfloor \left| \frac1x \right| + 1-\alpha \right\rfloor, \quad \text{if }x \neq 0.\]
In \cite{Nak81} Nakada studied the maps for $\alpha \in \big[\frac12,1 \big]$. He proved that each $T_\alpha$ has a unique ergodic absolutely continuous invariant measure $\nu_\alpha$ and he studied the dependence on $\alpha$ of the metric entropy $h_{\nu_\alpha}(T_\alpha)$. It was later found in various papers that the map $\alpha \mapsto h_{\nu_\alpha}(T_\alpha)$ has a very intricate structure. Here we mention specifically some relevant results from \cite{CT12,KSS12,BCIT13,CT13}.

\vskip .2cm
Let $a \in \mathbb Q \cap (0,1)$ with regular continued fraction expansion 
\[ a = [0; a_1a_2\cdots a_k] = [0;a_1 \cdots a_{k-1}(a_k-1)1].\]
From now on, we will always choose to write the expansion of $a$ that has an odd number of digits, so that $a = [0; a_1a_2\cdots a_{2n+1}]$. Associate to $a$ the interval $I_a \subseteq [0,1]$, given by
\[ I_a = (a^-, a^+) = ([0; (a_1 \cdots a_{2n+1})^{\infty}],[0; (a_1 \cdots a_{2n}(a_{2n+1}-1)1)^{\infty}])\]
if $a_{2n+1} \ge 2$ and
\[ I_a = (a^-, a^+) = ( [0; (a_1 \cdots a_{2n+1})^{\infty}], [0; (a_1 \cdots a_{2n-1}(a_{2n}+1))^{\infty}])\]
if $a_{2n+1}=1$. This is well defined by \eqref{q:gaussord}. The interval $I_a = (a^-,a^+)$ is called the {\em quadratic interval} with {\em pseudocenter} $a$ in \cite{CT12}. For $a=1$ this procedure does not work and we take $1^- = [0;1^\infty]$ and $1^+ = [0;1]$, so that $I_1 = (g,1)$, where $g = \frac{\sqrt 5 -1}{2}$ is the golden mean.

\vskip .2cm
In \cite[Proposition 2.4 and Lemma 2.6]{CT12} it was shown that any two quadratic intervals are either disjoint or one is a proper subset of the other. A quadratic interval is then called {\em maximal} if it is not properly contained in any other quadratic interval. The maximal quadratic intervals are the matching intervals for the family of maps $\{T_\alpha\}_{\alpha \in [0,1]}$, which in particular implies that the entropy $h_{v_\alpha}(T_\alpha)$ is monotone on $I_a$. More precisely, if $a = [0;a_1 \cdots a_{2n+1}]$ and if we set $N=1+\displaystyle\sum_{j \text{ even}}a_j $ and $M=-1 +\displaystyle\sum_{j \text{ odd}}a_j$, then in \cite[Theorem 3.8]{CT12} it is shown that on $I_a$,
\begin{equation}\label{q:hmonotone} h_{\nu_\alpha}(T_{\alpha}) \text{ is } \begin{cases}
\text{increasing } & \text{if } N-M<0,\\
\text{constant } & \text{if } N-M=0,\\
\text{decreasing } & \text{if } N-M>0.
\end{cases}\end{equation}
In \cite[Section 10]{KSS12} the authors showed that $\alpha \mapsto h_{\nu_\alpha}(T_\alpha)$ is constant on $[g, g^2]$ and their first open question asks whether this is a maximal value that is only attained on this interval.

\vskip .2cm
The results from \cite[Proposition 2.13 and Lemma 4.4]{CT12} characterise the maximal quadratic intervals. Since $a=[0;a_1\cdots a_{2n+1}]$ is assumed to have an odd length, their characterisation reduces to the following: $I_a$ is maximal if and only if 
\begin{equation}\label{q:max1}
a_1 \cdots a_{2n+1} <_g a_{i+1} \cdots a_{2n+1}a_1\cdots a_i \quad \text{for all } i=1,\ldots ,2n.
\end{equation}

\vskip .2cm
Recall the definition of $v$ from \eqref{q:decdot} and consider the map $\varphi:[0,1] \to \big[\frac12, 1\big]$ defined as follows. If $x \in [0,1]$ has regular continued fraction expansion $x = [0;a_1a_2a_3 \cdots]$, then
\[ \varphi(x) = v ( \underbrace{11\cdots 1}_{a_1} \underbrace{00\cdots 0}_{a_2} \underbrace{11\cdots 1}_{a_3} \cdots).\]
In \cite[Theorem 1.1]{BCIT13} the authors proved that $\varphi$ is an orientation reversing homeomorphism. Moreover, for the {\em bifurcation set}
\begin{equation}\label{q:E}
\mathcal E := (0,1) \setminus \bigcup_{\stackrel{a \in \mathbb Q \cap (0,1):}{I_a \text{ maximal}}} I_a
\end{equation}
of parameters $\alpha$ that are not contained in any maximal quadratic interval it was proven in \cite{BCIT13} that $\varphi(\mathcal E) = \Lambda$, where $\Lambda$ is the set from \eqref{q:lambda}.

\subsection{Identifying matching intervals and primitive words}
In this section we identify all matching intervals for the symmetric doubling maps by linking them to the maximal quadratic intervals of $\alpha$-continued fraction maps. Note that for any $a = [0; a_1 \cdots a_{2n+1}] \in \mathbb Q \cap (0,1)$ we have
\begin{equation}\label{q:rc}
\begin{split}
\varphi(a^-) =\ & v((1^{a_1}0^{a_2}\cdots 1^{a_{2n+1}}0^{a_1}1^{a_2}\cdots 0^{a_{2n+1}})^\infty),\\
\varphi(a^+) =\ & \begin{cases}
v((1^{a_1}0^{a_2}\cdots 1^{a_{2n+1}-1}0)^\infty), & \text{if } a_{2n+1} \ge 2,\\
v((1^{a_1}0^{a_2}\cdots 0^{a_{2n}+1})^\infty), & \text{if } a_{2n+1}=1.
\end{cases}\end{split}
\end{equation}
To $a$ we assign the word $w(a) \in \{0,1\}^*$ given by
\begin{equation}\label{q:wa}
w(a) = 1^{a_1}0^{a_2} \cdots 1^{a_{2n+1}}.
\end{equation}

\begin{prop}\label{p:primtomax}
For any $a = [0; a_1 \cdots a_{2n+1}] \in \mathbb Q \cap (0,1)$ the interval $\varphi(I_a)$ is given by
\[\varphi(I_a) = \big( \varphi(a^+), \varphi(a^-) \big) = \Big( \frac{2^m\varphi(a) -1}{2^m -1}, \frac{2^m \varphi(a)+1}{2^m +1} \Big) ,\]
where $m = \sum_{j=1}^{2n+1}a_j$ is the sum of the regular continued fraction digits of $a$.
\end{prop}

\begin{proof}
Write $w(a) =w_1 \cdots w_m$. Since $a^- = [0; (a_1 \cdots a_{2n+1})^{\infty}]$, we have
\begin{eqnarray*}
\varphi(a^-) &=& v \big( (1^{a_1}0^{a_2} \cdots 1^{a_{2n+1}}0^{a_1}1^{a_2} \cdots 1^{a_{2n+1}})^{\infty} \big)\\
&=& v \big( \big( w_1 \cdots w_m (1-w_1)\cdots (1-w_m)\big)^{\infty} \big).
\end{eqnarray*}
From the binary expansion of $\varphi(a^-)$, we see that $\varphi(a^-)$ is the fixed point of the map $D^{2m}$ corresponding to the digits $w_1 \cdots w_m (1-w_1) \cdots (1-w_m)$, so that $\varphi(a^-)$ is the unique solution to the equation
\begin{eqnarray*}
x= D^{2m}(x) &=& 2^{2m} x - w_1 2^{2m-1} - \cdots - 2^m w_m - 2^{m-1}(1-w_1) - \cdots -(1-w_m)\\
& = & 2^{2m} x - 2^m \varphi(a)(2^m-1) - (2^m-1).
\end{eqnarray*}
Hence,
\[ \varphi(a^-) = \frac{2^m \varphi(a) +1}{2^m+1}.\]
A similar calculation shows that
\[ \varphi(a^+) = v \big( \big( w_1 \cdots w_{m-1}0)^\infty \big) = \frac{2^m \varphi(a)-1}{2^m-1},\] since it is the fixed point of the map $D^m$ for the branch corresponding to the digits $w_1 \cdots w_{m-1}0$. The statement then follows from the fact that $\varphi$ is an orientation reversing homeomorphism.
\end{proof}

Recall that $I_1 =(g,1)$, where $g$ is the golden mean and note that $\varphi(1) = \frac12$ and
\[ \varphi(g) =  \sum_{k\ge 1} \frac1{2^{2k-1}} = \frac23.\]
Hence, $\frac1{\varphi(I_1)} = \big(\frac32,2\big)$. We know that for each $\eta \in \big(\frac32,2\big)$ the map $S_\eta$ has matching after one step; the interval $ \frac1{\varphi(I_1)} = \big(\frac32,2\big)$ is a matching interval. The next theorem says that in fact each maximal quadratic interval is mapped to a matching interval for the family of symmetric doubling maps by the map $x \mapsto \frac1{\varphi(x)}$.
\begin{theorem}\label{t:maxmatch}
Let $a=[0; a_1 \cdots a_{2n+1}] \in \mathbb Q \cap (0,1)$ and let $m = \sum_{j=1}^{2n+1} a_j$. Assume that $I_a$ is a maximal quadratic interval. Then any $\eta \in \frac1{\varphi(I_a)}$ has 
\[ d_{\eta,1}(1)\cdots d_{\eta,m-1}(1)=b_1\Big(\frac1{\eta}\Big) \cdots b_{m-1} \Big( \frac1{\eta} \Big) = w_1(a) \cdots w_{m-1}(a)
\]
and $m(\eta)=m$.
\end{theorem}

\begin{proof}
Assume that $a_{2n+1}=1$. The proof for the other case is similar. Write $w(a) =w_1 \cdots w_m$. From Proposition~\ref{p:primtomax} we know that
\[ \varphi(a^-) = v((w_1 \cdots w_m(1-w_1)\cdots (1-w_m))^\infty) \quad \text{and} \quad \varphi(a^+)= v((w_1\cdots w_{m-1}0)^\infty) . \]
Since both $\varphi(a^-)$ and $\varphi(a^+)$ have a purely periodic binary expansion, their binary expansion is unique. Recall the definitions of $\mathcal N$ and $\Lambda$ from (\ref{q:nonmatching}) and (\ref{q:lambda}) and the fact that $\mathcal N = \frac1{\Lambda \backslash \{ 0 \}}$. By \eqref{q:E} and the fact that $\varphi(\mathcal E) = \Lambda$ we have $\frac1{\varphi(a^-)}, \frac1{\varphi(a^+)} \in \mathcal N$ so that by Proposition~\ref{relationbinary} $D^k (\varphi(a^-)) \not \in \big( \frac12 \varphi (a^-), 1-\frac12 \varphi(a^-) \big)$ for any $k \ge 0$ and similarly, $D^k (\varphi(a^+)) \not \in \big( \frac12 \varphi (a^+), 1-\frac12 \varphi(a^+) \big)$ for any $k \ge 0$. Note that
\[ \Big(  \frac12 \varphi (a^+), 1-\frac12 \varphi(a^+) \Big) \subseteq \Big(\frac12 \varphi (a^-), 1-\frac12 \varphi(a^-) \Big).\]
Since the first $m-1$ binary digits of $\varphi(a^-)$ and $\varphi(a^+)$ coincide, i.e., since
\[ b_1(\varphi(a^-))\cdots b_{m-1}(\varphi(a^-)) = b_1(\varphi(a^+))\cdots b_{m-1}(\varphi(a^+)),\]
we have for any $0 \le k \le m-2$ that the points $D^k (\varphi(a^-))$ and $D^k(\varphi(a^+))$ either both lie to the left of $\frac12$ or both lie to right of $\frac12$. Hence, for these values of $k$ the set $D^k (\varphi(I_a))$ is an interval and
\begin{equation}\label{q:lambdadk}
\lambda \big( D^k(\varphi(I_a)) \big) = \lambda \big( \big(D^k (\varphi(a^-)), D^k (\varphi(a^+)) \big) \big) = 2^k \lambda (\varphi(I_a)),
\end{equation}
where $\lambda$ denotes the one-dimensional Lebesgue measure. For any $x \in \varphi(I_a)$ and $0 \le k \le m-2$ we have that $D^k (x) \in D^k (\varphi(I_a))$. If it holds that $D^k(\varphi(a^-)) \ge 1-\frac12 \varphi(a^-)$, then it follows from \eqref{q:lambdadk} $D^k(x)\ge 1-\frac12 x$. If on the other hand $D^k (\varphi(a^-))< \frac12 \varphi(a^-)$, then combining that $D^k(\varphi(a^+))< \frac12 \varphi(a^+)$ and \eqref{q:lambdadk}, we can also deduce that $D^k(x) \le \frac12 x$. So $D^k (x) \not \in \big( \frac12 x, 1-\frac12 x \big)$ for any $0 \le k \le m-2$.

\vskip .2cm
Next we consider the $(m-1)$-st iteration. By the periodicity and the form of the binary expansions of $\varphi(a^-)$ and $\varphi(a^+)$ we have that
\[ \varphi(a^+) = D^m (\varphi(a^+)) = 2D^{m-1}(\varphi(a^+)),\]
so $D^{m-1}(\varphi(a^+))= \frac12 \varphi(a^+)$ and similarly,
\[ 1-\varphi(a^-) = D^m (\varphi(a^-)) = 2D^{m-1}(\varphi(a^-))-1, \]
so $D^{m-1} (\varphi(a^-)) = 1- \frac12 \varphi(a^-)$. Hence, $D^{m-1}(\varphi(I_a)) = \big( \frac12 \varphi(a^+), 1-\frac12 \varphi(a^-) \big)$. Since $1-\frac12x > 1-\frac12 \varphi(a^-)$ we obviously have $D^{m-1}(x) < 1-\frac12 x$. The fact that $D^{m-1}(x) > \frac12x$ follows since $\lambda (D^{m-1}(\varphi(I_a)))=2^{m-1} \lambda (\varphi(I_a))$.
\end{proof}

From the previous theorem we get a complete description of the matching behaviour of $\{ S_\eta \}_{\eta \in [1,2]}$ in the sense of Theorem~\ref{t:mintervals}.
\begin{proof}[Proof of Theorem~\ref{t:mintervals}]
First note that if $\eta$ is not in the image of any maximal quadratic interval, then $\varphi^{-1}\big( \frac1{\eta} \big) \in \mathcal E$, where $\mathcal E$ is the bifurcation set from \eqref{q:E}. Since $\varphi(\mathcal E)= \Lambda$ and $\mathcal N = \frac1{\Lambda \setminus \{0\}}$, this means that $\eta \in \mathcal N$ and $S_\eta$ does not have matching.

\vskip .1cm
From Theorem~\ref{t:maxmatch} we know that if $\eta \in \frac1{\varphi(I_a)}$ for some maximal quadratic interval $I_a$, then $S_\eta$ has matching and the matching index $m$ is given by the sum of the regular continued fraction digits of $a$. Moreover, the first part of the signed binary expansion of 1 is equal to $w_1(a) \cdots w_m(a)$ for any $\eta \in \frac1{\varphi(I_a)}$. If, on the other hand, $\eta \not \in \frac1{\varphi(I_a)} \cup \mathcal N$, then there is another $a' \in \mathbb Q \cap (0,1]$ with $I_{a'}$ maximal and $\eta \in \frac1{\varphi(I_{a'})}$. This automatically implies that either $m(\eta) \neq m$ or
\[ d_{\eta,1}(1) \cdots d_{\eta,m}(1) \neq w_1(a) \cdots w_m(a). \qedhere\]
 \end{proof}
In other words, Theorem~\ref{t:maxmatch} says that the matching intervals for the family $\{ S_\eta \}_{\eta \in (1,2]}$ are precisely the images of the matching intervals for the family $\{ T_\alpha \}_{\alpha \in (0,1]}$ under the map $x \mapsto \frac1{\varphi(x)}$.

\begin{remark}{\rm
The map $x \mapsto \frac1{\varphi(x)}$ does \underline{not} give an isomorphism between the transformations $T_\alpha$ and $S_{\frac1{\varphi(\alpha)}}$, since that would imply that these maps have the same measure theoretical entropy and this is in general not the case. The measure theoretical entropy of $S_\eta$ is $\log 2$ for any parameter $\eta$, while there are several articles devoted to a detailed description of the dependence of the measure theoretic entropy of $T_\alpha$ on the parameter $\alpha$, such as \cite{NN08,CT12,KSS12,BCIT13,CT13}.
}\end{remark}

\vskip .2cm
From Theorem~\ref{t:maxmatch} we see that the matching intervals for $\{ S_\eta \}_{\eta \in (1,2]}$ are indexed by the initial part of the digit sequences $d_\eta$ up to the matching index $m(\eta)$ for the values of $\eta$ that are in the matching interval. We now characterise the initial words in $\{0,1\}^*$ that correspond to the matching intervals. Define the function $\psi : \{0,1\}^*\setminus \{\epsilon, 0, 1\} \to  \{0,1\}^*$ by
\begin{equation}\label{q:psi}
\psi (w) = \psi(w_1 \cdots w_m) = w_1 \cdots w_m(1-w_1)(1-w_2)\cdots(1-w_{m-1})1.
\end{equation}
Let $\sigma$ denote the left shift on sequences as before.
\begin{defn}\label{d:primitive}
A word $w= w_1 \cdots w_m \in \{0,1\}^*$ is called {\em primitive} if all the following hold:
\begin{itemize}
\item[(i)] $w_1=w_2 =w_m=1$;
\item[(ii)] $\sigma^n((w_1 \cdots w_m)^\infty) \preceq (w_1 \cdots w_m)^\infty$ for any $n \ge 0$;
\item[(iii)] there is no word $u \in \{0,1\}^*$ such that $u \prec w  \prec \psi(u)$.
\end{itemize}
\end{defn}
The conditions in the definition follow quite naturally from the dynamics of $S_\eta$. $w_1=w_2=1$ is necessary since $\eta \in \big(1, \frac32 \big)$ and $w_m=1$ since the last digit before matching is a 1. Condition (ii) is the usual restriction given by the dynamics of the system. In fact, any digit sequence $(d_{\eta,n}(x))_{n \ge 1}$ produced by the map $S_\eta$ will satisfy the following lexicographical condition: for any $k \ge 0$,
\[ \sigma^k (d_{\eta,n}(x))_{n \ge 1} \preceq (d_{\eta,n}(1))_{n \ge 1}.\]
Condition (iii) is the actual condition specifying which words correspond to a matching interval. It guarantees that matching occurs exactly at time $m$ and not before.

\vskip .2cm
The notion of primitivity is related, in fact equivalent, to the notion of admissibility which is central in the study of unique expansions (see Remark~\ref{r:othersets}).

\begin{defn}\label{d:admissible}
A word $w= w_1 \cdots w_m \in \{0,1\}^*$ is called {\em admissible} 
 if $m\ge 2$ and  
\[
(1-w_1)\cdots (1-w_{m-k})0^{\infty}\prec w_{k+1}\cdots w_m0^{\infty}\preceq w_1\cdots w_{m-k}0^{\infty}\quad\textrm{for all}\quad 0\le k\le m-1.
\]
\end{defn}

In \cite{K18}, Kong proved the equivalence of the above two notions. We state his result in the following proposition, and we give his proof in Appendix B.

\begin{prop}\label{equivalenceDerong}  A word $w\in \{0,1\}^*$ is primitive if and only if it is admissible.
\end{prop} 

Recall the definition of $w(a)$ from \eqref{q:wa} and associate similarly to each primitive $w \in \{0,1\}^*$ a rational number $a \in  \mathbb Q \cap (0,1)$ by setting
\[ a(w) = [0;a_1 \cdots a_{2n+1}] \quad \text{if  } w = 1^{a_1} 0^{a_2} \cdots 1^{a_{2n+1}}.\]
Using the fact that primitivity is equivalent to admissibility, we get the following result linking maximal quadratic intervals and primitive words.
\begin{prop}\label{maximaladmissible}
Let $w \in \{0,1\}^*$ and let $a \in \mathbb Q \cap (0,1)$. If $w$ is a primitive word, then $I_{a(w)}$ is a maximal quadratic interval and if $I_{a}$ is a maximal quadratic interval, then $w(a)$ is primitive.
\end{prop}

\begin{proof} First let $a = [0;a_1 \cdots a_{2n+1}] \in \mathbb Q \cap (0,1)$ be given and set $m=a_1+\cdots + a_{2n+1}$. Write $w(a) = w_1 \cdots w_m$. Suppose $I_{a}$ is maximal, so that (\ref{q:max1}) holds. For any $k$, 
\[w_{k+1}\cdots w_m=1^p 0^{a_{2\ell}}\cdots 1^{a_{2n+1}}, \text { or }\, w_{k+1}\cdots a_m=0^p 1^{a_{2\ell+1}}\cdots 1^{w_{2n+1}}.\]
In both cases, it follows directly from (\ref{q:max1}), that  
\[ w_{k+1}\cdots w_m0^{\infty}\preceq w_1\cdots w_{m-k}0^{\infty}\quad\textrm{for all}\quad 0\le k\le m-1.\]
So we only need to show that $(1-w_1)\cdots (1-w_{m-k})0^{\infty}\prec w_{k+1}\cdots w_m0^{\infty}$, and for this it suffices to consider the case $w_{k+1}\cdots w_m=0^p 1^{a_{2\ell+1}}\cdots 1^{a_{2n+1}}$. Now $p\le a_{2\ell}$ and $a_{2\ell}\le a_1$. If $p<a_1$, then the desired strict inequality holds and we are done. Assume $p=a_1$, then $a_{2\ell}=a_1$, and 
$w_{k+1}\cdots w_m=0^{a_{2\ell}} 1^{a_{2\ell+1}}\cdots 1^{a_{2n+1}}$. Assume for the sake of getting a contradiction that $w_{k+1}\cdots w_m\preceq(1-w_1)\cdots (1-w_{m-k})$. From (\ref{q:max1}) this implies $w_{k+1}\cdots w_m=(1-w_1)\cdots (1-w_{m-k})$, and hence $a_{2\ell}\cdots a_{2n+1}=a_1\cdots a_{2n-2\ell+2}$. By repeatedly applying this, we see that there exist integers $r$ and $s\le 2\ell -1$ such that
\[a_1\cdots a_{2n+1}=(a_1\cdots a_{2\ell-1})^ra_1\cdots a_s,\] 
and
\[a_{2\ell}\cdots a_{2n+1}a_1\cdots a_{2\ell-1}=(a_1\cdots a_{2\ell-1})^{r-1}a_1\cdots a_s(a_1\cdots a_{2\ell -1}).\]
Since 
\[ a_1\cdots a_{2n+1}  <_g a_{2\ell}\cdots a_{2n+1}a_1\cdots a_{2\ell-1},\]
we see that they differ in their last block of length $2\ell -1$, which occurs at the odd position $2n-2\ell +3$. This implies that $a_{s+1}\cdots a_{2\ell-1} a_1\cdots a_s <_g a_1\cdots a_{2\ell-1}$.
However, this contradicts the inequality $a_1\cdots a_{2n+1}<_g a_{2n-2\ell+3}\cdots a_{2n+1}a_1\cdots a_{2n-2\ell +2}$. Hence, $(1-w_1)\cdots (1-w_{m-k})0^{\infty}\prec w_{k+1}\cdots w_m0^{\infty}$, and $w(a)$ is admissible. Then $w(a)$ is primitive by Proposition~\ref{equivalenceDerong}.\\

Now let $w= w_1 \cdots w_m = 1^{a_1}0^{a_2}\cdots 1^{a_{2n+1}} \in \{0,1\}^*$ be a primitive word. By Proposition~\ref{equivalenceDerong} then $w$ is admissible. To prove that $I_{a(w)}$ is maximal, we verify (\ref{q:max1}) by considering strings starting with even and odd indices separately.

\medskip
\noindent
Let $k=a_1+\cdots +a_{2\ell-1}$. By admissibility we have $(1-w_1)\cdots (1-w_{m-k})\prec 0^{a_{2\ell}}\cdots 1^{a_{2n+1}}$, so there is a least $j\le 2n-2\ell+1$ such that $a_j\not= a_{2\ell+j-1}$. Admissibility also implies that $a_1\cdots a_j<_g a_{\ell}\cdots a_{2\ell+j-1}$ and hence $a_1\cdots a_{2n+1}<_g a_{2\ell}\cdots a_{2n+1}a_1\cdots a_{2\ell-1}$ as required.

\medskip
\noindent
We consider now the case $k=a_1+\cdots+ a_{2\ell}$. By admissibility  
\[w_{k+1}\cdots w_{m}=1^{a_{2\ell+1}}0^{a_{2\ell +2}}\cdots 1^{a_{2n+1}}\preceq w_1\cdots w_{m-k}.\]
If there exists a $1\le j\le 2n-2\ell$ such that $a_j\not=a_{2\ell+j}$, then admissibility would imply that $a_1\cdots a_{2n-2\ell}<_g a_{2\ell} \cdots a_{2n}$, and hence $a_1\cdots a_{2n+1}<_g a_{2\ell+1}\cdots a_{2n+1}a_1\cdots a_{2\ell}$. So assume that $a_j=a_{2\ell+j}$ for all $1 \le j \le 2n-2\ell$. By admissibility, we see that $w_1\cdots w_{m-k}=1^{a_1}\cdots 0^{a_{2n-2\ell}}1^t$ with $t = a_{2n+1}\le a_{2n-2\ell+1}$. If $a_{2n+1}< a_{2n-2\ell+1}$, then $a_1\cdots a_{2n-2\ell+1}<_g a_{2\ell} \cdots a_{2n+1}$ and hence $a_1\cdots a_{2n+1}<_g a_{2\ell+1}\cdots a_{2n+1}a_1\cdots a_{2\ell}$. We consider now the case $a_{2n+1}= a_{2n-2\ell+1}$. Using the equality $a_{2\ell+1}\cdots a_{2n+1}=a_1\cdots a_{2n-2\ell+1}$ repeatedly, we can write 
\[a_1\cdots a_{2n+1}=(a_1\cdots a_{2\ell})^ra_1\cdots a_s,\] 
and
\[a_{2\ell+1}\cdots a_{2n+1}a_1\cdots a_{2\ell}=(a_1\cdots a_{2\ell})^{r-1}a_1\cdots a_s(a_1\cdots a_{2\ell})\]
with $s\le 2\ell$. We see that they differ in their last block of length $2\ell$, which occurs at the even position $2n-2\ell +2$. Since by admissibility
\[(1-w_1)\cdots (1-w_k)=0^{a_1}1^{a_2}\cdots 1^{a_{2\ell}}\prec w_{m-k}\cdots w_m=0^{a_{2n-2\ell+2}}1^{a_{2n-2\ell+3}}\cdots 1^{a_{2n+1}},\]
this would imply that 
\[a_1\cdots a_{2\ell}<_g a_{s+1}\cdots a_{2n-2\ell +2}\cdots a_{2n+1}=a_{2\ell}a_1\cdots a_s\]
and hence
\[a_1\cdots a_{2n+1}=a_1\cdots a_{2n-2\ell+1}a_{2n-2\ell +2}\cdots a_{2n+1}<_g a_{2\ell+1}\cdots a_{2n+1}a_1\cdots a_{2\ell}\] 
as required. Therefore, $I_a$ is maximal.
\end{proof}

So, every matching interval for $\{ S_\eta \}_{\eta \in [1,2]}$ is coded by a primitive word. Next we show that the matching intervals exhibit a type of period doubling behaviour that was already observed numerically for the $\alpha$-continued fraction maps in \cite[Section 4.2]{CMPT10} and investigated further in \cite{BCIT13}.

\subsection{Cascades of matching intervals and non-matching parameters}
In the next two propositions we prove that attached to each matching interval we find a whole cascade of matching intervals separated by elements from the non-matching set $\mathcal N$ for which a Markov partition exists and with an accumulation point that is a transcendental element of $\mathcal N$. For a primitive word $w$, let $J_w = \frac1{\varphi(I_{a(w)})}$ be a matching interval and set
\[ L(w) = \frac1{\varphi(a^-(w))} \quad \text{and} \quad R(w) = \frac1{\varphi(a^+(w))},\]
so that $J_w = (L(w),R(w))$. Recall the definition of $\psi$ from \eqref{q:psi}.

\begin{prop}\label{p:cascade}
Let $w=w_1\cdots w_m$ be a primitive word. Then $L(w) = R(\psi(w))$.
\end{prop}

\begin{proof}
By \eqref{q:rc} and Proposition~\ref{p:primtomax}, $L(w) = \frac{2^m +1}{2^m v(w)+1}$ and $R(\psi(w)) = \frac{2^{2m} -1}{2^{2m}v(\psi(w))-1}$. Note that
\begin{eqnarray*}
2^{2m}v(\psi(w)) &=& w_1 2^{2m-1} + \cdots +w_{m-1} 2^{m+1} + 2^m + (1- w_1) 2^{m-1} + \cdots + (1-w_{m-1})2 + 1\\
& = & 2^m (2^m v(w))+2^m- 2^m v(w)=2^m(2^m v(w)+1)-2^m v(w).
\end{eqnarray*}
This implies 
\[ 2^{2m}v(\psi(w))-1=2^m(2^mv(w)+1)-(2^mv(w)+1)=(2^m-1)(2^mv(w)+1).\]
Hence,
\[ \frac{2^{2m}-1}{2^{2m}v(\psi(w))-1} = \frac{(2^m+1)(2^m-1)}{(2^m-1)(2^mv(w)+1)} = \frac{2^m+1}{2^mv(w)+1 }. \qedhere\]
\end{proof}

\noindent So, attached to each matching interval is a cascade of matching intervals corresponding to the words $\psi^n(w)$. Set $\underline{w} := \lim_{n \to \infty} \psi^n(w)$. Note that $\underline{w}$ does not depend on where in the cascade we start. Write $p_w =  v(\underline{w})$,
then
\[ \lim_{n \to \infty} L\big( \psi^n(w)\big) = \lim_{n \to \infty} \frac{2^{2^nm}+1}{2^{2^nm}v(\psi^n(w))+1} = \lim_{n \to \infty} \frac{1+ \frac1{2^{2^nm}}}{v(\psi^n(w)) + \frac1{2^{2^nm}}} = \frac1{p_w}.\]

\begin{ex}\label{x:tm}
Consider the primitive word $11$. Then $J_{11} = \big( \frac54, \frac32)$, so for each $\eta \in J_{11}$, $S_{\eta}$ has matching after 2 steps. We also have $\psi(11)=1101$ and $J_{1101} = \big( \frac{17}{14}, \frac54 \big)$. For any $\eta$ in this interval, $S_{\eta}$ has matching after four steps. For $\eta=\frac54$ we have $1-\eta = -\frac14$ and
\[ \begin{array}{llll}
S_{\frac54}(1) = \frac34, & S_{\frac54}^2(1) = \frac14, & S_{\frac54}^3(1) = \frac12, & S_{\frac54}^4(1)=1.\\
\\
S_{\frac54}\big(1-\frac54\big) = -\frac12, & S_{\frac54}^2\big(-\frac14\big)=-1, & S_{\frac54}^3\big(-\frac14\big) = -\frac34, & S_{\frac54}^4\big(-\frac14\big) = -\frac14.
\end{array}\]
The limit $\underline{11} = \lim_{n \to \infty} \psi^n(11)$ is the shifted Thue-Morse sequence. Recall that the Thue-Morse substitution is given by
\[ 0 \mapsto 01, \quad 1 \mapsto 10.\]
The Thue-Morse sequence is the fixed point of this substitution, which is
\[ t = 0110 \, 1001 \, 10010110 \, 1001011001101001 \cdots,\]
and the Thue-Morse constant $p^*$ is the number that has this sequence as its base 2 expansions, i.e., $p^* = \sum_{n \ge 1} \frac{t_n}{2^n} \approx 0.412454$. The limit sequence $\underline{11}$ is the sequence obtained when we shift the Thue-Morse sequence one place to the left, $\underline{11}=\sigma(t)$. For the corresponding constant we get
\[ \lim_{n \to \infty} L\big( \psi^n(11) \big) = \frac1{p_{11}} = \frac{1}{2p^*}  \approx 1.212216,\]
which is transcendental (see \cite{Dek77}).
\end{ex}

\vskip .2cm
The previous example illustrates a general pattern. Let $w=w_1 \cdots w_m \in \{0,1\}^m$ be a primitive word and consider the right end point $R(w)$ of the matching interval $J_w$. Then $d_{R(w)} = (w_1 \cdots w_{m-1}0)^{\infty}$ by \eqref{q:rc} and $R(w) \in \mathcal N$. It holds that $S_{R(w)}^m(1)=1$ and thus $S_{R(w)}^m(1-R(w))=1-R(w)$. In other words, $S_{R(w)}$ has a Markov partition, but no matching.
Results from \cite{BCIT13} give us the following proposition.

\begin{prop}
Let $w \in \{0,1\}^m$ be a primitive word. Then $\frac1{p_w} \in \mathcal N$ and $\frac1{p_w}$ is transcendental.
\end{prop}

\begin{proof}
Since $\mathcal N$ contains exactly those points that are not in any matching interval $J_w$, we have $\frac1{p_w} = \lim_{n \to \infty} L(\psi^n(w)) \in \mathcal N$. In \cite{BCIT13} the maps $\tau_0$ and $\Delta$ are defined on words $u=u_1\cdots u_n \in \{0,1\}^*$ by  $\tau_0(u) = v ((u_1 \cdots u_n0)^{\infty})$ and $\Delta(u)= u_1 \cdots u_n1 (1-u_1)\cdots (1-u_n)$. Then for each $j \ge 1$ they set $\tau_j(u) = \tau_0(\Delta^j(u))$ and $\tau_\infty(u) = \lim_{j \to \infty} \tau_j (u)$. It is proven in \cite[Proposition 4.7]{BCIT13} that if $\tau_0(u) \in \Lambda$, then $\tau_\infty(u)$ is transcendental. The fact that $\frac1{p_w}$ is transcendental thus follows from \cite[Proposition 4.7]{BCIT13} by observing that $p_w = \tau_\infty (w_1 \cdots w_{m-1})$.
\end{proof}

\begin{remark}\label{r:othersets}{\rm
We have already seen that $\mathcal N = \frac1{\varphi(\mathcal E)\setminus \{0\}}$, where $\mathcal E$ is the bifurcation set for the family $\{ T_\alpha \}_{\alpha \in [0,1]}$ defined in \eqref{q:E}. In \cite{BCIT13} the set $\mathcal E$ was linked to the real slice of the boundary of the Mandelbrot set, the set $\Lambda$ of codings for the kneading invariants of unimodal maps encoded by $\Lambda$, and the set of univoque bases. So, this article adds a new item to this list of correspondences. The common link between the results from \cite{BCIT13} and our case is the set $\Gamma$ from (\ref{q:gamma}), which was shown by Allouche and Cosnard (\cite{AC83,AC01}) to have connections with univoque numbers. Given a number $1< \beta < 2$, one can express any real number $x \in \big[0, \frac1{\beta-1} \big]$ as a {\em $\beta$-expansion}: $x = \sum_{n \ge 1} \frac{c_n}{\beta^n}$ for some sequence $(c_n)_{n \ge 1} \in \{0,1\}^{\mathbb N}$. Typically a number $x$ has uncountably many different expansions of this form. The number $1 < \beta <2$ is called {\em univoque} if there is a unique sequence $(c_n)_{n \ge 1} \in \{0,1\}^{\mathbb N}$ such that $1= \sum_{n \ge 1}\frac{c_n}{\beta^n}$, i.e., if 1 has a unique $\beta$-expansion. Let $\mathcal U$ denote the set of univoque bases. The properties of $\mathcal U$ were studied by many authors. There exists an equivalent characterisation of univoque numbers in terms of strongly admissible sequences, which is mainly due to Parry (\cite{Par60}), see also \cite{EJK90}, and which is stronger than the notion of admissibility from Definition~\ref{d:admissible}.

\vskip .2cm
A sequence $(c_n)_{n \ge 1} \in \{0,1\}^{\mathbb N}$ is {\em strongly admissible} if and only if
\[ \left\{
\begin{array}{ll}
c_{k+1} c_{k+2} \cdots \prec c_1c_2 \cdots, & \text{if } c_k=0,\\
\\
c_{k+1}c_{k+2}\cdots \succ c_1c_2 \cdots, & \text{if } c_k=1,
\end{array}\right.\]
for all $k$.  It is easy to check that strong admissibility is equivalent to 
\[(1-c_1)(1-c_2)\cdots \prec c_{k+1}c_{k+2}\cdots \prec c_1c_2 \cdots\]
for all $k\ge 1$. In \cite{EJK90} it is proved that a sequence $(c_n)_{n \ge 1}$ is strongly admissible if and only if there is a univoque $\beta>1$ such that $1 = \sum_{n \ge 1} \frac{c_n}{\beta^n}$. In \cite{KL07}, the authors studied the topological structure of the set $\mathcal U$ as well as its closure $\overline{\mathcal U}$. Furthermore they characterised the sequences that belong  to $\overline{\mathcal U}$, namely $(d_n)_{n \ge 1}\in \overline{\mathcal U}$ if and only if
\[(1-d_1)(1-d_2)\cdots \prec d_{k+1}d_{k+2}\cdots \preceq d_1d_2 \cdots\]
for all $k\ge 1$. They also showed that if $(d_n)_{n \ge 1}\in \overline{\mathcal U}$, then $(d_n)_{n \ge 1}$ has an arbitrarily long initial block that is admissible in the sense of Definition \ref{d:admissible}. In \cite{AC01} the authors showed that there is a one-to-one correspondence between strong admissible sequences and the points in $\Gamma$ that do not have a periodic binary expansion. In this way the univoque bases are related to a subset of $\mathcal N$.
}\end{remark}

\section{The continuity of the frequency function}
In this section we consider the unique absolutely continuous invariant measure $\mu_\eta$ and the associated {\em frequency function} $\eta \mapsto \mu_{\eta} \big(\big[-\frac12, \frac12 \big] \big)$ with the goal of giving an explicit formula for $\mu_{\eta} \big(\big[-\frac12, \frac12 \big] \big)$ on each of the matching intervals $J_w$. Recall the expression for the invariant probability density from (\ref{q:density}). Suppose that for some $\eta$ we have matching after $m$ steps. Hence, $S^m_\eta (1)=S^m_\eta(1-\eta)$ and $S^m_\eta (\eta-1) = S^m_\eta (-1)$. Moreover, before matching we have $S^n_\eta (1) = S^n_\eta (1-\eta) + \eta$ and $S^n_\eta (\eta-1) = S^n_\eta (-1) + \eta$. Then
\[ f_\eta(x) = \frac1C \sum_{n = 0}^{m-1} \frac{1}{2^{n+1}} \Big(1_{[S^n_\eta (1-\eta), S^n_\eta (1))}(x) + 1_{[S^n_\eta(-1), S^n_\eta (\eta-1))}(x) \Big),\]
where $C$ is the normalising constant. $C$ is related to the total measure, which is
\begin{eqnarray*}
\mu_\eta([-1,1]) &=& \frac1{C}\int_{-1}^1 f_\eta(x)dx = \frac2C \sum_{n=0}^{m-1} \frac1{2^{n+1}} \big( S^n_\eta(1)-S^n_\eta(1-\eta) \big)\\
&=& \frac{\eta}C \Big( \frac{1-\frac1{2^m}}{1-\frac12} \Big) = \frac{2\eta}{C}\Big(1-\frac1{2^m} \Big)=1,
\end{eqnarray*}
where we have used that $S^n_\eta (1) = S^n_\eta (1-\eta) + \eta$ before matching. Hence,
\begin{equation}\label{q:C}
\frac1C =\frac1{\eta} \frac{2^{m-1}}{2^m-1}.
\end{equation}
We first prove that the map $\eta \mapsto f_\eta$ is continuous.
\begin{theorem}\label{t:convL1}
Let $\bar \eta \in [1,2]$ and let $\{ \eta_k \}_{k \ge 1} \subseteq [1,2]$ be a sequence converging to $\bar \eta$. Then $f_{\eta_k} \to f_{\bar \eta}$ in $L^1$.
\end{theorem}

For the proof we use the Perron-Frobenius operator $\mathcal L_\eta$ for $S_\eta$, which is given by
\[ (\La f)(x) = \frac12 \Big( f\Big( \frac{x}{2} \Big) + 1_{(\eta-2, \eta-1)}(x) f\Big( \frac{x-\eta}{2}\Big) +1_{(1-\eta, 2-\eta)}(x) f\Big( \frac{x+\eta}{2}\Big) \Big)\]
for $f \in L^1(\lambda)$, where as before $\lambda$ denotes the one-dimensional Lebesgue measure. For a function $f:[-1,1]\to \mathbb R$, let $Var(f)$ denote its total variation and use $BV$ to denote the set of functions $f:[-1,1]\to \mathbb R$ of bounded variation, so with $Var(f)<\infty$. Since $S_\eta$ is an expanding interval map with constant slope, it is well known, see for example \cite[Theorem 5.2.1]{BG}, that an absolutely continuous invariant density for $S_\eta$ is obtained from the Lebesgue almost everywhere limit of a subsequence of $\big( \frac1n \sum_{j=0}^{n-1} \mathcal L_\eta^j(1) \big)_{n \ge 1}$.

\begin{proof}[Proof of Theorem~\ref{t:convL1}]
First let $\eta \neq \frac32$. Let $\mathcal P_\eta$ be the collection of intervals of monotonicity of $S_\eta^2$. To be precise, for $\eta \in \big[1, \frac32 \big)$ we have
\[ \mathcal P_\eta = \Big\{ \Big(-1, -\frac{2\eta+1}{2} \Big), \Big( -\frac{2\eta+1}{4}, -\frac12 \Big), \Big( -\frac12, -\frac14 \Big), \Big( -\frac14, \frac14 \Big), \Big( \frac14, \frac12 \Big), \Big( \frac12, \frac{2\eta+1}{4} \Big), \Big( \frac{2\eta+1}{4},1 \Big) \Big\},\]
and for $\eta \in \big(\frac32,2 \big]$ this becomes
\[ \mathcal P_\eta = \Big\{ \Big(-1, \frac{1-2\eta}{2} \Big), \Big( \frac{1-2\eta}{4}, -\frac12 \Big), \Big( -\frac12, -\frac14 \Big), \Big( -\frac14, \frac14 \Big), \Big( \frac14, \frac12 \Big), \Big( \frac12, \frac{2\eta-1}{4} \Big), \Big( \frac{2\eta-1}{4},1 \Big) \Big\}.\]
Since $S_\eta$ is a piecewise expanding interval map with constant slope
it follows from \cite[Lemma 5.2.1]{BG} that for $f\in BV$,
\begin{eqnarray*}
 Var(\mathcal L_\eta f) & \le &  Var (f) + 2 \int_{[-1,1]} |f| d\lambda,\\
 Var (\mathcal L^2_\eta f) & \le & \frac12 Var (f) + \frac1{2\delta(\eta)} \int_{[-1,1]} |f| d\lambda,
\end{eqnarray*}
where $\delta (\eta) = \min \{ \lambda(I)\, : \, I \in \mathcal P_\eta \}= \big| \frac{2\eta -3}{4} \big|$. For $n \ge 2$, write $n=2j+i$ with $i \in \{0,1\}$. Then
\begin{equation}\label{q:varn}
 \begin{split}
Var (\mathcal L_\eta^n f) =\ & Var (\mathcal L^j_{S_\eta^2} \mathcal L_\eta^i f) \le  \frac1{2^j} Var(\mathcal L_\eta^i f) + \sum_{k=0}^{j-1} \frac1{2^k} \frac1{2\delta(\eta)} \int_{[-1,1]} |f| d\lambda\\
\le \ & \frac1{2^j} Var (f) + \int_{[-1,1]}|f| d\lambda \Big( 2 + \frac1{\delta(\eta)} \Big).
\end{split}
\end{equation}
Fix some $\bar \eta \in [1,2] \setminus \{ \frac32\}$. Let
\[ 0 < \varepsilon <  \min \Big\{ \Big|\bar \eta - \frac32\Big|,  \delta(\bar \eta) \Big\}.\]
Then for all $\eta \in [\bar \eta -\varepsilon, \bar \eta + \varepsilon]$, we have $\delta(\eta) \ge \delta(\bar \eta) - \frac12 \varepsilon \ge \frac{\delta(\bar \eta)}{2}$. For $k\ge 1$, define the sequence of functions $f_{k,n} = \frac1n \sum_{j=0}^{n-1} \mathcal L_{\eta_k}^j (1)$. Recall that there is a subsequence of $(f_{k,n})$ converging to $f_{\eta,k}$ $\lambda$-a.e. For ease of notation we just assume that $\lim_{n \to \infty} f_{k,n}= f_{\eta_k}$ $\lambda$-a.e. By \eqref{q:varn},
\[ Var (f_{k,n} ) \le \frac1n \sum_{j=0}^{n-1} Var (\mathcal L^j_{\eta_k} (1)) \le 4 + \frac2{\delta(\eta_k)},\]
so for all $k$ sufficiently large, we have
\[ Var (f_{k,n} ) \le 4 + \frac2{\delta( \bar \eta) - \frac12 \varepsilon} \le 4 + \frac4{\delta(\bar \eta)}.\]
Also,
\[ \sup | f_{k,n} | \le Var(f_{k,n}) + \int_{[-1,1]} f_{k,n} d\lambda \le Var(f_{k,n})  + \frac1n \sum_{j=0}^{n-1} \int_{[-1,1]}  \mathcal L^j_{\eta_k} (1) d\lambda \le 6 + \frac4{\delta( \bar \eta)}.\]
Since both of these bounds are independent of $\eta_k$ and $n$, we have $Var (f_{\eta_k}), \sup |f_{\eta_k}| \le 6 + \frac4{\delta( \bar \eta)}$ for each $k$ large enough. From Helly's Theorem it then follows that there is a subsequence $\{k_i\}$ and an $f_{\infty} \in BV$ such that $f_{\eta_{k_i}} \to f_{\infty}$ in $L^1(\lambda)$ and $\lambda$-a.e.~and with $\sup |f_{\infty}| , Var(f_\infty) \le  6 + \frac4{\delta( \bar \eta)}$. We show that $f_\infty = f_{\bar \eta}$ by proving that for each Borel set $B \subseteq [-1,1]$ we have
\begin{equation}\label{q:hinfty}
\int_B f_\infty d\lambda = \int_{S_{\bar \eta}^{-1}(B)} f_\infty d\lambda.
\end{equation}
The desired result then follows from the uniqueness of the invariant probability density. First note that $1_B \in L^1(\lambda)$, so it can be approximated arbitrarily closely by compactly supported $C^1$ functions. So instead of (\ref{q:hinfty}) we prove that
\[ \big| \int_{[-1,1]} (g \circ S_{\bar \eta}) f_\infty d\lambda - \int_{[-1,1]} g \, f_\infty d\lambda \big| =0\]
for any compactly supported $C^1$ function $g$ on $[-1,1]$. (Hence $\| g\|_{\infty} < \infty$.) We split this into three parts:
\begin{align*}
\big| \int_{[-1,1]} (g \circ S_{\bar \eta}) f_\infty d\lambda - \int_{[-1,1]} g\, f_\infty d\lambda \big| & \le \big| \int_{[-1,1]} (g \circ S_{\bar \eta}) f_\infty d\lambda - \int_{[-1,1]} (g \circ S_{\bar \eta})f_{\eta_{k_i}} d\lambda \big|\\
& + \big| \int_{[-1,1]} (g \circ S_{\bar \eta}) f_{\eta_{k_i}} d\lambda - \int_{[-1,1]} (g \circ S_{\eta_{k_i}})f_{\eta_{k_i}} d\lambda \big|\\
& + \big| \int_{[-1,1]} (g \circ S_{\eta_{k_i}}) f_{\eta_{k_i}} d\lambda - \int_{[-1,1]} g \, f_\infty d\lambda \big|.
\end{align*}
Then for the first part we have
\begin{align*}
\big| \int_{[-1,1]} (g \circ S_{\bar \eta}) f_\infty d\lambda - \int_{[-1,1]}& (g \circ S_{\bar \eta})f_{\eta_{k_i}} d\lambda \big|\\
&\le \big| \int_{[-1,1]} \big(\sup_{x \in [-1,1]} g(x) \big) (f_\infty - f_{\eta_{k_i}}) d\lambda\\
& \le \| g \|_\infty \int_{[-1,1]} |f_\infty - f_{\eta_{k_i}}| d\lambda = \|g \|_\infty \| f_\infty - f_{\eta_{k_i}} \|_{L^1}.
\end{align*}
For the third part we get
\begin{eqnarray*}
\big| \int_{[-1,1]} (g \circ S_{\eta_{k_i}}) f_{\eta_{k_i}} d\lambda - \int_{[-1,1]} g \, f_\infty d\lambda \big| & = & \big| \int_{[-1,1]} g \, f_{\eta_{k_i}} d\lambda - \int_{[-1,1]} g \, f_\infty d\lambda \big|\\
&\le & \| g \|_\infty \| f_\infty - f_{\eta_{k_i}} \|_{L^1}.
\end{eqnarray*}
Hence, these two parts converge to 0 as $i \to \infty$. Now, for the middle part we have
\begin{align*}
\big| \int_{[-1,1]} (g \circ S_{\bar \eta}) f_{\eta_{k_i}} d\lambda - \int_{[-1,1]} & (g \circ S_{\eta_{k_i}})f_{\eta_{k_i}} d\lambda \big|\\
& \le  \int_{[-1,1]} \big| (g\circ S_{\bar \eta}) - (g\circ S_{\eta_{k_i}}) \big| f_{\eta_{k_i}} d\lambda\\
& \le  \big( \sup_{x \in [-1,1]} f_{\eta_{k_i}}(x) \big) \int_{[-1,1]} \big| (g\circ S_{\bar \eta}) - (g\circ S_{\eta_{k_i}}) \big| d\lambda\\
& \le  \lambda([-1,1])  \int_{[-1,1]} \big| (g\circ S_{\bar \eta}) - (g\circ S_{\eta_{k_i}}) \big| d\lambda.
\end{align*}
We split this integral into three parts again, now according to the intervals of monotonicity of $S_{\bar \eta}$. Note that by the Dominated Convergence Theorem and by the continuity of $g$, we have
\[ \lim_{i \to \infty} \int_{[-1,-1/2)} \big| (g\circ S_{\bar \eta}) - (g\circ S_{\eta_{k_i}}) \big| d\lambda = \int_{-1}^{-1/2} \lim_{i \to \infty} \big| g(2x + \eta_{k_i}) - g(2x+\eta) \big| dx =0.\]
Similarly, we can prove that the integral converges to 0 on $\big[-\frac12, \frac12 \big]$ and $\big( \frac12, 1\big]$. Hence, $f_\infty = f_{\bar \eta}$ Lebesgue a.e.
In fact, the proof shows that for each subsequence of $(f_{\eta_k})$ there is a further subsequence that converges a.e.~(and in $L^1$) to $f_{\bar \eta}$. This is equivalent to saying that the sequence $( f_{\eta_k})$ converges in measure to $f_{\bar \eta}$. Since $f_{\eta_k}\le 6 + \frac4{\delta( \bar \eta)}$ for any $k$ large enough, this implies that $(f_{\eta_k})$ is uniformly integrable from a certain $k$ on, so by Vitali's Theorem the sequence $(f_{\eta_k})$ converges in $L^1$ to $f_{\bar \eta}$.

\vskip .2cm
The above proof shows that $\eta \mapsto f_\eta$ is continuous at any point $\bar \eta \in [1,2]\setminus \{ \frac32 \}$. For $\frac32$, we have by \eqref{q:h2} that
\[ \begin{split}
\lim_{\eta \downarrow \frac32} \int_{[-1,1]} |f_\eta & -f_{\frac32}|d\lambda\\
  =\ & \lim_{\eta \downarrow \frac32} \Big( 2 \int_{-1}^{1-\eta} \Big(\frac13 - \frac1{2\eta} \Big)d x + 2 \int_{1-\eta}^{-\frac12} \Big( \frac1{\eta}-\frac13 \Big) dx  + \int_{-\frac12}^{\frac12} \Big( \frac23 - \frac1{\eta} \Big) dx \Big)=0.
\end{split}\]
For any $\eta$ in the matching interval $J_{11} = (\frac43, \frac32)$ that is close enough to $\frac32$ we have
\[ f_\eta = \frac1{6\eta} \big( 2 + 1_{(2-2\eta, 2\eta-2)} + 1_{(\eta-2, 2-\eta)} + 2\cdot 1_{(1-\eta, \eta-1)} \big),\]
which is easily checked by direct computation.  Then
\[ \begin{split}
\lim_{\eta \uparrow \frac32} \int_{[-1,1]}  |f_\eta - f_{\frac32}|d \lambda
 = \ & \lim_{\eta \uparrow \frac32} \Big( 2 \int_{-1}^{2-2\eta} \Big(\frac13 - \frac1{3\eta} \Big) d x + 2 \int_{2-2\eta}^{\eta-2} \Big( \frac1{2\eta} - \frac13 \Big) d x \\
& + 2 \int_{\eta-2}^{-\frac12} \Big( \frac2{3\eta}-\frac13 \Big) d x + 2 \int_{-\frac12}^{1-\eta} \Big( \frac23 - \frac2{3\eta} \Big) d x + \int_{1-\eta}^{\eta-1} \Big( \frac23 - \frac1{\eta} \Big) d x \Big) =0.
\end{split}\]
This gives the result also for $\eta = \frac32$.
\end{proof}

\begin{cor}\label{c:measure0}
The frequency function $\eta \mapsto \mu_{\eta} \big(\big[-\frac12, \frac12 \big] \big)$ is continuous.
\end{cor}

\begin{proof}
This immediately follows from the previous result, since for any $\eta \in [1,2]$ and any sequence $\{ \eta_k \}_{k \ge 1} \subseteq [1,2]$ converging to $\eta$, we have
\[ \begin{split}
\lim_{k \to \infty} \big|\mu_{\eta} \Big(\Big[-\frac12, \frac12 \Big] \Big) - \mu_{\eta_k}\Big(\Big[-\frac12, \frac12 \Big] \Big) \big| \le \ & \lim_{k \to \infty} \int_{[-\frac12, \frac12]} |f_{\eta} - f_{\eta_k}| d\lambda\\
 \le \ & \lim_{k \to \infty} \int_{[-1, 1]} |f_{\eta} - f_{\eta_k}| d\lambda =0. \qedhere
 \end{split}\]
\end{proof}

We now give a precise description of the measure of the middle interval $\big[ -\frac12, \frac12 \big]$. The fact that before matching $S^k_\eta (1) = S^k_\eta (1-\eta) + \eta$ in particular implies that $S^k_\eta(1)$ will only visit $\big[-\frac12, \frac12 \big]$ and $\big( \frac12, 1\big]$ and $S^k_\eta (1-\eta)$ will only visit $\big[-1, -\frac12 \big)$ and $\big[-\frac12, \frac12 \big]$. Moreover, $S^k_\eta(1)$ and $S^k_\eta (1-\eta)$ will never both be in $\big[-\frac12, \frac12 \big]$. We also know that matching occurs immediately after $S^k_\eta(1) \in \big(\frac12, 1 \big]$ and $S^k_\eta (1-\eta) \in \big[-1, -\frac12 \big)$. So, up to one step before matching we always have exactly one of the two orbits in $\big[-\frac12, \frac12 \big]$.

\vskip .2cm
Let $w = w_1 \cdots w_m$ be a primitive word and $\eta \in J_w$. In order to determine $\mu_\eta\big( \big[-\frac12, \frac12 \big] \big)$ we need to describe functions of the form $x \mapsto 1_{[S^k_\eta (1-\eta), S^k_\eta(1))}(x) 1_{[-\frac12, \frac12]}(x)$. Note that for $0 \le k \le m-2$ if $w_{k+1}=1$, then
\[ 1_{[S^k_\eta (1-\eta), S^k_\eta (1))}(x) 1_{[-\frac12, \frac12]}(x) = 1_{[S^k_\eta (1-\eta), \frac12]}(x)\]
and if $w_{k+1}=0$, then 
\[ 1_{[S^k_\eta (1-\eta), S^k_\eta (1))}(x) 1_{[-\frac12, \frac12]}(x) = 1_{[-\frac12, S^k_\eta(1))}(x).\]
Moreover, 
\[ 1_{[S^{m-1}_\eta (1-\eta), S^{m-1}_\eta(1))}(x) 1_{[-\frac12, \frac12]}(x) = 1_{[-\frac12, \frac12]}(x).\]
Due to symmetry, the measure of $[-\frac12, \frac12]$ is then given by
\begin{eqnarray*}
\mu_\eta\Big(\Big[-\frac12, \frac12\Big]\Big) &=& \frac1C \int_{[-\frac12, \frac12]} f_\eta(x) dx \\
&=& \frac1C \sum_{\stackrel{0\le k \le m-2:}{w_{k+1=1}}} \frac1{2^k} \Big( \frac12 - S^k_\eta (1-\eta) \Big) +\frac1C\sum_{\stackrel{0\le k \le m-2:}{w_{k+1=0}}} \frac1{2^k} \Big( S^k_\eta(1) +\frac12 \Big) + \frac1{C 2^{m-1}}\\
&=& \frac1C \Big( \frac1{2^{m-1}} + \sum_{k=0}^{m-2} \frac1{2^{k+1}} + \sum_{\stackrel{0\le k \le m-2:}{w_{k+1=1}}} \Big( \frac{\eta}{2^k} - \frac1{2^k}S^k_\eta(1) \Big) + \sum_{\stackrel{0\le k \le m-2:}{w_{k+1=0}}} \frac1{2^k} S^k_\eta(1) \Big).
\end{eqnarray*}
Observe that if $m=1$, the above two summations are zero. In this case by \eqref{q:C}
\[ \mu_\eta\Big(\Big[-\frac12, \frac12 \Big]\Big)=\frac1C = \frac1{\eta},\] 
as we saw before. For $m\ge 2$, we get
\[ \mu_\eta\Big(\Big[-\frac12, \frac12\Big]\Big) = \frac1C \Big( 1+ \sum_{\stackrel{0\le k \le m-2:}{w_{k+1=1}}} \Big( \frac{\eta}{2^k} - \frac1{2^k}S^k_\eta(1) \Big) + \sum_{\stackrel{0\le k \le m-2:}{w_{k+1=0}}} \frac1{2^k} S^k_\eta(1) \Big).\]
If $m=2$, then by \eqref{q:C} we find $\frac1C = \frac2{3\eta}$ and $w_1=1$, so for all $\eta\in J_{11}$,
\[ \mu_\eta\Big(\Big[-\frac12, \frac12 \Big]\Big) = \frac1C (1+ \eta -1)= \frac23.\]
Assume $m\ge 3$. Recall the notation $w_i^j = w_i \cdots w_j$ and also recall that for $k \le m$ we have $S^k_\eta (1) = 2^k-2^k v(w_1^k)\eta$. Then
\[ \mu_\eta\Big(\Big[-\frac12, \frac12\Big]\Big)  
= \frac1C \Big( \eta + \sum_{\stackrel{1\le k \le m-2:}{w_{k+1=1}}} \Big( \frac{\eta}{2^k} - 1 +v(w_1^k)\eta \Big) +\sum_{\stackrel{1\le k \le m-2:}{w_{k+1=0}}} \Big( 1 - v(w_1^k)\eta\Big) \Big).\]
Let
\[ \eta(w) = \# \{ 2\le k \le m-1\, : \, w_k=0 \} - \# \{ 2\le k\le m-1 \, : \, w_k=1 \}. \]
Then
\begin{equation}\label{q:anothermu}
\begin{split}
\mu_\eta\Big(\Big[-\frac12, \frac12\Big]\Big) =\ & \frac1C \Big(\eta+\eta(w) + \eta \Big(
\sum_{\stackrel{1\le k \le m-2:}{w_{k+1=1}}} \Big( \frac1{2^k} + v(w_1^k) \Big) - \sum_{\stackrel{1\le k \le m-2:}{w_{k+1=0}}} v(w_1^k) \Big)\\
=\ & \frac{2^{m-1}}{2^m-1} \Big( \frac{\eta(w)}{\eta} + 2v(w_1^{m-1}) + \sum_{\stackrel{1\le k \le m-2:}{w_{k+1=1}}} v(w_1^k) - \sum_{\stackrel{1\le k \le m-2:}{w_{k+1=0}}} v(w_1^k) \Big),
\end{split}
\end{equation}
where we have used that $\displaystyle \sum_{\stackrel{1\le k \le m-2:}{w_{k+1=1}}} \frac1{2^k}= \sum_{1 \le k \le m-2} \frac{w_{k+1}}{2^k} =2v(w_1^{m-1})-1$. If $w_{k+1}=1$, then $v(w_1^{k+1})=v(w_1^k) + \frac1{2^{k+1}}$ and if $w_{k+1}=0$, then $v(w_1^{k+1}) =v(w_1^k)$. This gives
\[\mu_\eta\Big(\Big[-\frac12, \frac12\Big]\Big) 
= \frac{2^{m-1}}{2^m-1} \Big( \frac{\eta(w)}{\eta} + v(w_1^{m-1})  + \sum_{\stackrel{1\le k \le m-1:}{w_k=1}} v(w_1^k) - \sum_{\stackrel{2\le k \le m-1:}{w_k=0}} v(w_1^k) \Big).\]
For any primitive word $w_1 \cdots w_m$ ($m\ge 3$)  the expression
\[ K_w:= v(w_1^{m-1}) +\sum_{\stackrel{1\le k \le m-1:}{w_k=1}} v(w_1^k)  - \sum_{\stackrel{1\le k \le m-1:}{w_k=0}}  v(w_1^k)\]
has a constant value on $J_w$. As a result 
\begin{equation}\label{q:measure0}
\mu_\eta\Big(\Big[-\frac12, \frac12\Big]\Big) =\frac{2^{m-1}}{2^m-1} \Big( \frac{\eta(w)}{\eta} +K_{w}\Big),
\end{equation}
as a function of $\eta$ on the interval $J_{w}$, is increasing for $\eta(w)<0$, decreasing for $\eta(w)>0$, and constant if $\eta(w)=0$. Note that using \eqref{q:measure0} we can for any $\eta \in \big(1,\frac32\big) \setminus \mathcal N$ explicitly calculate the frequency of the digit 0 in the signed binary expansion $ (d_{\eta,n}(x))_{n \ge 1}$ for typical $x$. We now have Theorem~\ref{t:23}.
\begin{proof}[Proof of Theorem~\ref{t:23}]
The continuity of $\eta \mapsto \mu_\eta\big( \big[ -\frac12, \frac12 \big] \big)$ is proved in Corollary~\ref{c:measure0} and the other two statements of the theorem follow from \eqref{q:measure0}.
\end{proof}

\section{The maximal frequency of the digit 0}

It remains to show that the map $\eta \mapsto \mu_\eta\big(\big[-\frac12, \frac12\big]\big)$ takes its maximal value $\frac23$ exactly on the interval $\big[ \frac65, \frac32 \big]$. Figure~\ref{f:maximalinterval} shows the value of $\mu_\eta\big(\big[-\frac12, \frac12\big]\big)$ for several primitive words $w$. We can obtain the value on the big plateau in the middle easily from the relation with the $\alpha$-continued fraction maps as follows.

\vskip .2cm
From \eqref{q:measure0} it follows that the monotonicity behaviour of $\eta \mapsto \mu_\eta\big(\big[-\frac12, \frac12\big]\big)$ equals that of the entropy function $\alpha \mapsto h_{\nu_\alpha}(T_{\alpha})$ for $\alpha$-continued fractions as described in \eqref{q:hmonotone}. To see this, let $w=1^{a_1}\cdots 1^{a_{2n+1}}$ be a primitive word. Then by Proposition \ref{maximaladmissible} the quadratic interval $I_a$ is maximal with $a=[0;a_1 \cdots a_{2n+1}]$. For $N=1+\displaystyle\sum_{j \text{ even}}a_j$ and $M=-1 +\displaystyle\sum_{j \text{ odd}}a_j$ we have $N-M=\eta(w)$. Therefore, by \eqref{q:hmonotone} we get that $\mu_\eta\big(\big[-\frac12, \frac12\big]\big)$ increases, decreases or is constant on $J_w$ if and only if $h_{\nu_\alpha}(T_\alpha)$ increases, decreases or is constant on $I_{a(w)}$ respectively. As mentioned in Section~\ref{ss:alpha} it was shown in \cite[Section 10]{KSS12} that the entropy function $\alpha \mapsto h_{\nu_\alpha}(T_{\alpha})$ is constant on the interval $[g,g^2]$, where $g=\frac{\sqrt5 -1}{2}$. Since $\varphi(g)=\frac32$ and $\varphi(g^2)=\frac65$, we have that $\mu_\eta\big(\big[-\frac12, \frac12\big]\big) = \mu_{\frac32}\big(\big[-\frac12, \frac12\big]\big) =\frac23$ for any $\eta \in \big[ \frac65, \frac32 \big]$. This gives the plateau on the top of the graph in Figure~\ref{f:maximalinterval}.

\vskip .2cm
To prove Theorem~\ref{t:max} we still need to show that $\mu_\eta\big(\big[-\frac12, \frac12\big]\big) < \frac23$ for all other values of $\eta$. We know from \eqref{q:>32} that this holds for $\eta > \frac32$. Before we start on the proof of Theorem~\ref{t:max}, we consider the set of primitive words of the form $1^m$ separately in the next example.

\begin{figure}
\includegraphics[height=.2\textheight]{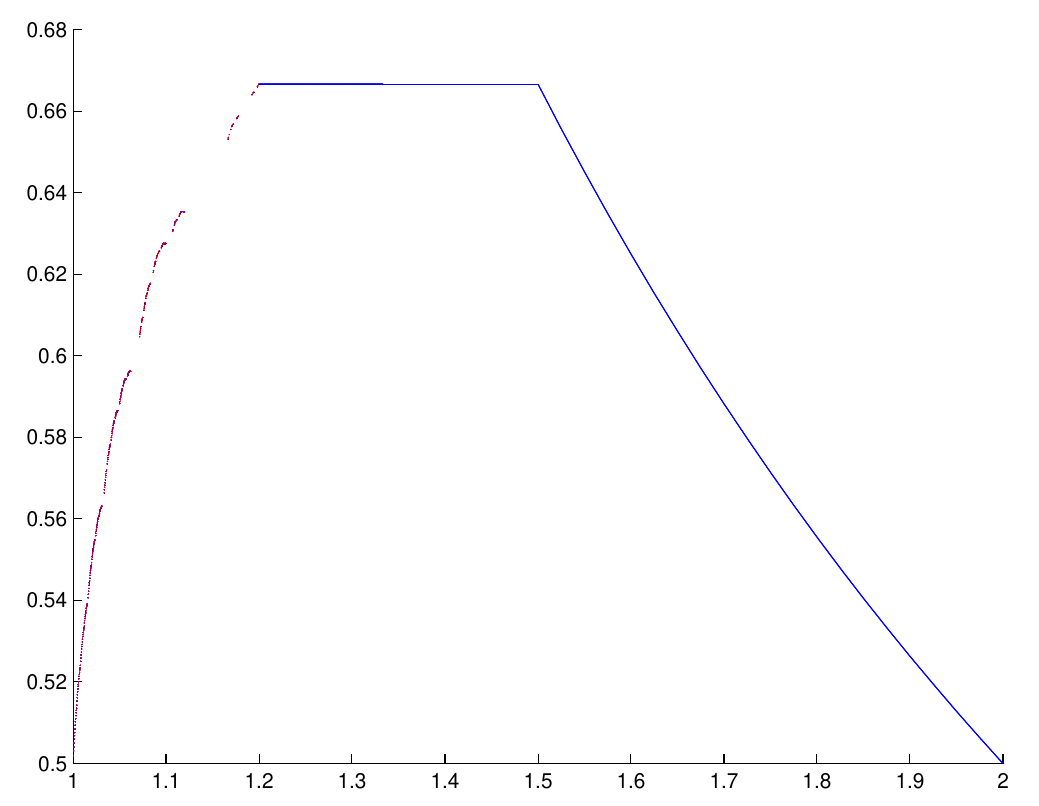}
\caption{The value of $\mu_\eta\big(\big[-\frac12, \frac12\big]\big)$ on all matching intervals $J_w$ with $m\le 13$. The picture is made by Niels Langeveld.}
\label{f:maximalinterval}
\end{figure}

\begin{ex}\label{x:all1}
Let $w=1^m$ with $m \ge 3$. Then $\eta(w) = -(m-2)$, so that the maximal value attained for $\mu_\eta \big( \big[ -\frac12, \frac12 \big] \big)$ is attained when $\frac1{\eta}$ is smallest, i.e., if
\[ \frac1{\eta} = v((1^{m-1}0)^\infty) = \frac12 \frac1{1-\frac1{2^m}} + \cdots + \frac1{2^{m-1}} \frac1{1-\frac1{2^m}} = \frac1{2^m-1} \sum_{k=1}^{m-1} 2^k = \frac{2^m-2}{2^m-1}.\]
For $1 \le n \le m-1$ we have $v(w_1^n) = \sum_{k=1}^{n} \frac1{2^k} = 1-\frac1{2^n}$. Thus, for $\eta \in J_w$ it holds by \eqref{q:measure0} that
\[ \begin{split}
\mu_\eta \Big( \Big[ -\frac12, \frac12 \Big] \Big) \le \ & \frac{2^{m-1}}{2^m-1} \Big( -\frac{(m-2)(2^m-2)}{2^m-1} + 1 - \frac1{2^{m-1}} + \sum_{n=1}^{m-1} \Big( 1-\frac1{2^n} \Big) \Big)\\
=\ & \frac{2^{m-1}}{2^m-1} \Big( -m+2 + \frac{m-2}{2^m-1} + 1- \frac1{2^{m-1}} +m-1 - 2+\frac1{2^{m-1}} \Big)\\
= \ & \frac{2^{m-1}}{2^m-1} \frac{m-2}{2^m-1}.
\end{split}\]
Note that this is less than $\frac23$ if and only if $3(m-2)2^{m-1} < 2(2^m-1)^2$, which is if and only if 
\[ m(2^m+2^{m-1}) + 2^{m+2} < 2^{2m+1} + 2^{m+1} + 2^m +2.\]
This obviously holds for $m=3$. For $m>3$ we have $2m-1>m+2$ and of course $m < 2^m$, so
\[  m(2^m+2^{m-1}) + 2^{m+2} < 2^{2m} + 2^{2m-1} + 2^{m+2} < 2^{2m+1} < 2^{2m+1} + 2^{m+1} + 2^m +2.\]
Hence, also for $m>3$ and any $\eta \in J_{1^m}$ we have that $\mu_\eta \big( \big[ -\frac12, \frac12 \big] \big)< \frac23$. Then the same statement holds on the cascade of matching intervals attached to $J_{1^m}$, i.e., for any $\eta \in J_{\psi^n(1^m)}$, $n \ge 1$. 
\end{ex}

The proof of Theorem~\ref{t:max} requires the following lemma on the value of $\mu_\eta\big(\big[-\frac12, \frac12\big]\big)$ on the cascades of matching intervals.
\begin{lemma}\label{l:constant}
Let $w$ be a primitive word. Then $\eta(\psi^n(w)) = 0$ for all $n \ge 1$ and the frequency function $\eta \mapsto \mu_\eta \big( \big[ -\frac12, \frac12 \big]\big)$ is constant on the interval $\big[p_w, L(w) \big]$.
\end{lemma}

\begin{proof}
Let $w=w_1 \cdots w_m$ be a primitive word and let $k$ denote the number of 0's that occur in $w$. Then $\eta(w) = k-(m-2-k) = 2k+2-m$. Since $w_m=1$ and $1-w_1 =0$, the number of 0's in $\psi(w)$ is $k+1+(m-2-k)=m-1$. Hence, $\eta(\psi(w)) = 0$ and by the same reasoning $\eta(\psi^n(w))=0$ for all $n$. By (\ref{q:measure0}) $\eta \mapsto \mu_\eta \big( \big[ -\frac12, \frac12 \big]\big)$ is constant on each interval $J_{\psi^n(w)}$, so that by continuity
\[ \mu_\eta \Big( \Big[ -\frac12, \frac12 \Big]\Big) = \mu_{L(w)} \Big( \Big[ -\frac12, \frac12 \Big]\Big)\]
for all $\eta \in \big[p_w, L(w) \big]$.
\end{proof}

\begin{proof}[Proof of Theorem~\ref{t:max}.]
To prove Theorem~\ref{t:max} it is enough to consider $\eta \in \big(1, \frac65 \big)$. Corollary~\ref{c:measure0} then implies that we only need to consider values of $\eta$ in a matching interval $J_w$ for which $w$ is primitive and $w0^\infty \succ 1(10)^\infty$. The formula obtained in \eqref{q:measure0} shows that whether on $J_w$ the frequency map $\eta \mapsto \mu_\eta \big( \big[ -\frac12, \frac12 \big]\big)$ is increasing, decreasing or constant, depends on the value of $\eta(w)$. By Proposition~\ref{p:cascade} and Lemma~\ref{l:constant} it follows that attached to each matching interval there is a cascade of matching intervals on which the frequency map is constant. Hence, by Corollary~\ref{c:measure0}, to prove that $\mu_\eta \big( \big[ -\frac12, \frac12 \big]\big) < \frac23$ for any $\eta \in \big(1, \frac65 \big)$ it is enough to prove that this is the case for any $\eta$ in a matching interval $J_w$ that satisfies $\eta(w)=0$. This is what we will do.

\vskip .2cm
Let $w = w_1 \cdots w_m = 1^{a_1} 0^{a_2} \cdots 1^{a_{2n+1}}$ be a primitive word with $\eta(w)=0$. Then $m$ is even and
\[ a_1 + \cdots + a_{2n+1}-2=\frac{m-2}{2} = a_2 + \cdots + a_{2n}.\]
Moreover, by Example~\ref{x:all1} we can assume that $w=1^{a_1} 0 w_{a_1+2} \cdots w_m$, so that $m \ge a_1+2$. To give precise estimates on $\mu_\eta \big( \big[ -\frac12, \frac12 \big]\big)$ we rewrite \eqref{q:anothermu} to obtain that
\begin{equation}\label{q:newmeasure0}
\mu_\eta \Big( \Big[ -\frac12, \frac12 \Big] \Big) =  \frac{2^{m-1}}{2^m-1} \Big( 2v(w_1^{m-1}) + \sum_{\stackrel{1 \le \ell \le m-2:}{w_{\ell+1=1}} } \Big( v(w_1^\ell) - \frac1{\eta} \Big) + \sum_{\stackrel{1 \le \ell \le m-2:}{w_{\ell+1=0}} } \Big( \frac1{\eta} - v(w_1^\ell) \Big)\Big).
\end{equation}
Set $s_0=0$ and for $1 \le k \le 2n+1 $, $s_k = a_1 + \cdots + a_k $. Since $\mu_\eta \big( \big[ -\frac12, \frac12 \big]\big)$ is the same for all $\eta \in J_w$, we take $\eta = \frac1{v(w)}$. Note that for any $k$ the $k$-th block of 1's in $w$ satisfies
\[ \sum_{j =s_{2k}+1}^{s_{2k+1}} \frac1{2^j} = \frac1{2^{s_{2k}}} - \frac1{2^{s_{2k+1}}} = \frac{2^{a_{2k+1}}-1}{2^{s_{2k+1}}},\]
so that
\[ \sum_{\stackrel{1 \le \ell \le m-2:}{w_{\ell+1=0}} }  \Big( v(w) - v(w_1^\ell) \Big) = \sum_{j=1}^n \frac{2^{a_{2j+1}}-1}{2^{s_{2j+1}}}(a_2 + \cdots + a_{2j}).\]
On the other hand,
\[ \begin{split}
\sum_{\stackrel{1 \le \ell \le m-2:}{w_{\ell+1=1}} } & \Big( v(w) - v(w_1^\ell) \Big)\\
 =\ & \sum_{2 \le \ell+1 \le s_1} \frac1{2^\ell}v(w_{\ell+1}^m) + \sum_{s_2+1 \le \ell+1 \le s_3} \frac1{2^\ell}v(w_{\ell+1}^m) + \cdots + \sum_{s_{2n}+1 \le \ell \le s_{2n+1}-1} \frac1{2^\ell} v(w_{\ell+1}^m)\\
=\ & 2\Big( v(w) - \frac12 - \frac1{2^m} \Big) - \frac{a_1-1}{2^{s_1}} - \frac{a_3}{2^{s_3}} - \cdots - \frac{a_{2n-1}}{2^{s_{2n-1}}} - \frac{a_{2n+1}-1}{2^{s_{2n+1}}}\\
& + \sum_{j=1}^n (a_1 + a_3 + \cdots + a_{2j-1}-1) \frac{2^{a_{2j+1}}-1}{2^{s_{2j+1}}}\\
=\ & 2v(w)-1-\frac1{2^{m-1}} + \sum_{j=1}^n \frac{(a_1+a_3+\cdots +a_{2j-1}-1)(2^{a_{2j+1}}-1)-a_{2j+1}}{2^{s_{2j+1}}} - \frac{a_1-1}{2^{a_1}}.
\end{split}\]
Inserting this into \eqref{q:newmeasure0} and using that $2v(w_1^{m-1}) = 2v(w)-\frac1{2^{m-1}}$ we get that the inequality
\[ \mu_\eta \Big( \Big[ -\frac12, \frac12 \Big] \Big)
= \frac{2^{m-1}}{2^m-1}\Big[ 1 + \frac{a_1-1}{2^{a_1}} + \sum_{j=1}^n \frac{a_{2j+1} + (2^{a_{2j+1}}-1) (1+ \sum_{\ell=1}^{2j} (-1)^\ell a_\ell)}{2^{s_{2j+1}}} \Big] < \frac23\]
holds if and only if
\begin{equation}\label{q:ineq}
2^{m+1} +  2^m + 2^2 + 3 (a_1-1) 2^{m-a_1} + 3\sum_{j=1}^n R_j(w) < 2^{m+2},
\end{equation}
where we have put
\[ R_j(w) = a_{2j+1} 2^{m-s_{2j+1}}+ 2^{m-s_{2j}}\Big(1+ \sum_{\ell=1}^{2j} (-1)^\ell a_\ell \Big) + 2^{m-s_{2j+1}} \Big( \sum_{\ell=1}^{2j}(-1)^{\ell+1} a_\ell-1\Big).\]
We split the proof into several cases according to the value of $a_1$.

\vskip .2cm
\noindent {\bf (I)} $a_1 \ge 3$: Consider first the term $3 (a_1-1) 2^{m-a_1}$. For $a_1=4$ and $a_1=3$ we can just calculate this and for $a_1 \ge 5$ we can use the fact that $p \le 2^{\frac{p}{2}}$ for any $p \ge 4$ to obtain the following:
\begin{equation}\label{q:est1}
3 (a_1-1) 2^{m-a_1} \le \begin{cases}
2^{m-2} + 2^{m-3}, & \text{if } a_1 \ge 5,\\
2^{m-1}+2^{m-4}, & \text{if }a_1=4,\\
2^{m-1}+2^{m-2}, & \text{if }a_1=3.
\end{cases}\end{equation}
Next consider $R_1(w)$. By primitivity we have $a_1>a_2$. For $a_1=3$ we can make precise estimates by considering the cases $a_2=1$ and $a_2=2$ separately. For $a_2=1$ we have $a_1-a_2-1=1$ and using the fact that $p \le 2^{p-1}$ for all $p \ge 1$ gives
\[
3R_1(w) = 3(a_32^{m-s_3}-2^{m-s_2}+2^{m-s_3})\le 3(2^{m-s_3+a_3} - 2^{m-s_2}) = 0.
 \]
In case $a_2=2$ we get $a_1-a_2-1=0$ and thus
\[ 3R_1(w) = 3 a_3 2^{m-a_3} \le 3 \cdot 2^{m-s_2-1} = 2^{m-5} + 2^{m-6}.\]
For $a_1 \ge 4$, we get again by using that $p \le 2^{p-1}$ for all $p \ge 1$,
\[ 
3R_1(w) \le  3\cdot 2^{m-s_3}(a_1-a_2+a_3-1) \le 3 \cdot 2^{m-2a_2-2} \le 2^{m-3}+2^{m-4}.
\]
So, we have obtained that
\begin{equation}\label{q:j=1}
3R_1(w) \le \begin{cases}
2^{m-3} + 2^{m-4}, & \text{if } a_1\ge 4,\\
2^{m-5} + 2^{m-6}, & \text{if } a_1 =3.
\end{cases}
\end{equation}
Now consider $R_j(w)$ for $2 \le j \le n$. Note that the terms $1+\sum_{\ell=1}^{2j}(-1)^\ell a_\ell $ and $\sum_{\ell=1}^{2j}(-1)^{\ell+1} a_\ell-1$ have opposite signs, so that we only need to consider one of them. First assume that $1+ \sum_{\ell=1}^{2j}(-1)^\ell a_\ell >0$. Then, since $p \le 2^{p-1}$ for all $p \ge 1$ and $s_k \ge a_1 + k-1$ for all $k$, we get
\[ \begin{split}
R_j(w) \le \ & a_{2j+1}2^{m-s_{2j+1}} + 2^{m-s_{2j}} \Big( 1+ \sum_{\ell=1}^{2j}(-1)^\ell a_\ell \Big) \\
\le & 2^{m-s_{2j}-1} + 2^{m-2(a_1+a_3+\cdots + a_{2j-1})} \le 2^{m-a_1-2j} + 2^{m-2a_1-2j+2}.
\end{split}\]
On the other hand, if $-1 + \sum_{\ell=1}^{2j}(-1)^{\ell+1} a_\ell>0$, then
\[ \begin{split}
R_j(w) \le \ & a_{2j+1}2^{m-s_{2j+1}} + 2^{m-s_{2j+1}}\Big( -1+ \sum_{\ell=1}^{2j}(-1)^{\ell+1} a_\ell \Big)\\
\le \ & 2^{m-s_{2j}-1} + 2^{m-2(a_2+a_4+ \cdots +a_{2j}+1)} \le 2^{m-a_1-2j} + 2^{m-2j-2}.
\end{split}\]
So, in all cases we have $R_j(w) \le 2^{m-a_1-2j} + 2^{m-2j-2}$. Then, since $a_1 \ge 3$,
\begin{equation}\label{q:sum} \begin{split}
3\sum_{j=2}^n R_j(w) \le 3 \sum_{j=2}^n (2^{m-3-2j}+2^{m-2-2j}) = 3\sum_{j=6}^{3+2n}2^{m-j} < 3 \cdot 2^{m-5} = 2^{m-4}+2^{m-5}.
\end{split}\end{equation}
Putting \eqref{q:est1}, \eqref{q:j=1} and \eqref{q:sum} together leads to the following. For $a_1 \ge 5$ it follows from $\eta(w)=0$ that $m \ge 10$, so that
\[ \begin{split}
2^{m+1} + 2^m+2^2 & +3(a_1-1)2^{m-a_1}+3\sum_{j=1}^n R_j(w)\\
 \le \ & 2^{m+1} + 2^m+2^2 + 2^{m-2} + 2^{m-3} + 2^{m-3} + 2^{m-4} + 2^{m-4}+2^{m-5}\\
=\ & 2^{m+1} + 2^m + 2^{m-1} + 2^{m-3} +  2^{m-5} + 2^2 < 2^{m+2}.
 \end{split}\]
For $a_1 =4$ we have $m \ge 8$ and we obtain
\[ \begin{split}
2^{m+1} + 2^m+2^2 & +3(a_1-1)2^{m-a_1}+3\sum_{j=1}^n R_j(w)\\
 \le \ & 2^{m+1} + 2^m+2^2 + 2^{m-1}+2^{m-4} + 2^{m-3} + 2^{m-4} + 2^{m-4} + 2^{m-5}\\
 =\ & 2^{m+1} + 2^m + 2^{m-1} + 2^{m-2} +2^{m-4} + 2^{m-5}+ 2^2 < 2^{m+2}.
\end{split}\]
For $a_1 =3$ we have $m \ge 6$ and get
\[ \begin{split}
2^{m+1} + 2^m+2^2 & +3(a_1-1)2^{m-a_1}+3\sum_{j=1}^n R_j(w)\\
 \le \ & 2^{m+1} + 2^m+2^2 + 2^{m-1}+2^{m-2} + 2^{m-5} + 2^{m-6} + 2^{m-4} + 2^{m-5}\\
 =\ & 2^{m+1} + 2^m + 2^{m-1} + 2^{m-2} +2^{m-3} + 2^{m-6}+ 2^2 < 2^{m+2}.
\end{split}\]
Hence, in all cases the inequality from \eqref{q:ineq} is satisfied and we have $\mu_\eta \big( \big[ -\frac12, \frac12 \big]\big) < \frac23$.

\vskip .2cm
\noindent {\bf (II)} $a_1=2$: In this case $w$ starts with 1101 and cannot have more than two consecutive 0's or 1's. Since $\eta < \frac65$ we have $w0^\infty \succ 1(10)^\infty$, so there is some $k \ge 1$, such that
\[ w= 1(10)^k110 w_{2k+5} \cdots w_m.\]
Since $a_{2j+1}2^{m-s_{2j}+1} = 2^{m-s_{2j}-1}$ for all $1 \le j \le n$ and $a_1=a_2+1$ we have that
\[ R_j(w) = 2^{m-s_{2j}-1} + 2^{m-s_{2j}} (1-2^{-a_{2j+1}})(-a_3+a_4- \cdots - a_{2j-1}+a_{2j}).\]
We further know that $3(a_1-1)2^{m-a_1}=2^{m-1}+2^{m-2}$. Since for any $2 \le j \le k$ we have $-a_3 + a_4 - \cdots -a_{2j-1}+a_{2j} =0$, it follows that
\[ 3\sum_{j=2}^k R_j(w) = 3\sum_{j=2}^k 2^{m-s_{2j}-1} = \sum_{j=5}^{2k+2}2^{m-j}.\]
Hence, to show that inequality \eqref{q:ineq} holds we would like to have an upper bound for
\[ Q(w) = \sum_{j=-1}^{2k+2} 2^{m-j}+2^2 + 3\sum_{j=k+1}^n \big( 2^{m-s_{2j}-1} + 2^{m-s_{2j}} (1-2^{-a_{2j+1}})(-a_3+a_4- \cdots - a_{2j-1}+a_{2j}) \big).
\]
The terms $-a_3 + a_4 - \cdots -a_{2j-1}+a_{2j}$ are largest if we see blocks 00 before we see blocks 11. Since $\psi(1(10)^k11) = 1(10)^k110(01)^{k+1}$, primitivity implies that each occurrence of 00 must be followed by $(10)^k$ or $(10)^j11$ for some $j < k$. This in turn implies that there is no primitive word $w$ with $\eta(w)=0$ and $m < 2(2k+3)$. Hence, we consider only words of even length at least equal to $2(2k+3)$ and for each such word $w$ there is a unique $p\ge 1$, such that its length $m$ satisfies $(p+1)(2k+3) \le m < (p+2)(2k+3)$. We consider a number of cases.

\vskip .2cm
\noindent If $p$ is odd, then
\[ Q(w) \le \begin{cases}
Q\big( 1(10)^k11 ( 0(01)^{k+1})^p \big), & \text{if } m = (p+1)(2k+3),\\
Q\big( 1(10)^k 11 (0 (01)^{k+1})^p 0 (01)^i 1\big), & \text{if } m = (p+1)(2k+3) + 2i+2, \, 0 \le i \le k.
\end{cases}\]
If $p$ is even instead, then 
\[ Q(w) \le \begin{cases}
Q\big( 1(10)^k11 ( 0(01)^{k+1})^p 1\big), & \text{if } m = (p+1)(2k+3)+1,\\
Q\big( 1(10)^k 11 (0 (01)^{k+1})^p 0 (01)^i \big), & \text{if } m = (p+1)(2k+3) + 2i+1, \, 0 \le i \le k.
\end{cases}\]
We begin by computing an upper bound for $Q(w^\diamondsuit)$, where $w^\diamondsuit=1(10)^k11(0(01)^{k+1})^p$.
It holds that $a_{2j+1}=1$ for all $j \ge k+1$, so that we can rewrite $Q(w^\diamondsuit)$ as
\begin{equation}\label{q:R2}
Q(w^\diamondsuit) = \sum_{j=-1}^{2k+2} 2^{m-j}+2^2 + 3\sum_{j=k+1}^n 2^{m-s_{2j}-1}(1-a_3 + a_4 - \cdots -a_{2j-1}+a_{2j}).
\end{equation}
For each $1 \le j \le p$ and $0 \le \ell \le k$ we have
\[ -a_3+a_4- \cdots -a_{2jk+2j+2\ell-1}+a_{2jk+2j+2\ell} = j-1\]
and $s_{2jk+2j+2\ell} = 2jk+3j+2\ell+3$. Putting this in \eqref{q:R2} gives
\[ \begin{split}
Q(w^\diamondsuit) =\ & \sum_{j=-1}^{2k+2} 2^{m-j}+2^2 + 3\sum_{j=1}^p j \sum_{\ell=0}^k 2^{m-2jk-3j-2\ell-3}\\
= \ & \sum_{j=-1}^{2k+2} 2^{m-j}+2^2 + 3 \sum_{\ell=0}^k 2^{m-2\ell-3} \sum_{j=1}^p j  2^{-(2k+3)j}.
\end{split}\]
Since
\[  \sum_{j=1}^p j  2^{-(2k+3)j} < \sum_{j=1}^\infty (j+1)2^{-(2k+3)} - \sum_{j=1}^\infty 2^{-(2k+3)j} = \frac{2^{2k+3}}{(2^{2k+3}-1)^2} < \frac1{2^{2k+2}},\]
we have
\begin{equation}\label{q:sumj}
3 \sum_{\ell=0}^k 2^{m-2\ell-3} \sum_{j=1}^p j  2^{-(2k+3)j} < 2^{m-2k-5} \Big(4- \frac1{2^{2k}}\Big) < 2^{m-2k-3}.
\end{equation}
From $p \ge 1$ it follows that $m = (p+1)(2k+3) \ge 2(2k+3)$, so that
\[ Q(w^\diamondsuit) < \sum_{j=-1}^{2k+3} 2^{m-j}+2^2 < 2^{m+2}.\]
Now consider $Q(w^\spadesuit)$, where $w^\spadesuit = 1(10)^k 11 (0 (01)^{k+1})^p 0 (01)^i 1$ for some $0 \le i \le k$. Then $a_{2j+1}=1$ for all $k+1 \le j < n$ and $a_{2n+1}=2$. With the same computations as above this gives
\[ \begin{split}
Q(w^\spadesuit) =\ &  \sum_{j=-1}^{2k+2} 2^{m-j}+2^2 + 3\sum_{j=1}^p j \sum_{\ell=0}^k 2^{m-2jk-3j-2\ell-3} + 3\sum_{\ell=0}^{i-1} (p+1) 2^{m-2(p+1)k-3(p+1)-2\ell-3}\\
& + 3 \cdot 2^{m-2(p+1)k-3(p+1)-2i-3} + 3p \cdot 2^{m-2(p+1)k-3(p+1)-2i-2} (1-2^{-2})\\
\le \ & \sum_{j=-1}^{2k+2} 2^{m-j}+2^2 + 3\sum_{j=1}^{p+1} j \sum_{\ell=0}^k 2^{m-2jk-3j-2\ell-3} + 3p\cdot 2^{m-2(p+1)k-3(p+1)-2i-4}.
\end{split}\]
Using \eqref{q:sumj} and the fact that $p \le 2^{p-1}$ we get that
\[
Q(w^\spadesuit) <  \sum_{j=-1}^{2k+3} 2^{m-j}+2^2 +  2^{m-2(p+1)k-2p-2i-7} + 2^{m-2(p+1)k-2p-2i-8}.
\]
Since $p \ge 1$ we obtain that $2(p+1)k+2p+2i+7 \ge 4k+ 9 >2k+5$ and it follows that $Q(w^\spadesuit) < 2^{m+2}$. Next consider $Q(w^\heartsuit)$, where $w^\heartsuit = 1(10)^k11 ( 0(01)^{k+1})^p 1$, so that $m=2pk + 3p +2k +4$. Then again we have $a_{2j+1}=1$ for all $k+1 \le j < n$ and $a_{2n+1}=2$. In the same way as above this gives
\[ \begin{split}
Q(w^\heartsuit) =\ & \sum_{j=-1}^{2k+2} 2^{m-j}+2^2 + 3\sum_{j=1}^{p-1} j \sum_{\ell=0}^k 2^{m-2jk-3j-2\ell-3} + 3\sum_{\ell=0}^{k-1} p 2^{m-2pk-3p-2\ell-3}\\
& + 3 \cdot 2^{m-2pk-3p-2k-3} + 3(p-1) \cdot 2^{m-2pk-3p-2k-2} (1-2^{-2})\\
<\ &  \sum_{j=-1}^{2k+3} 2^{m-j}+2^2 + 3(p-1)\cdot 2^{m-2pk-3p-2k-3}.
\end{split}\]
Since $p \ge 2$, we use $p-1 \le 2^{p-2}$ to obtain that
\[ Q(w^\heartsuit) < \sum_{j=-1}^{2k+3} 2^{m-j}+2^2 +  2^{m-2pk-2p-2k-4} +  2^{m-2pk-2p-2k-5} < 2^{m+2}.\]
Finally, consider $Q(w^{\clubsuit})$, where $w^\clubsuit = 1(10)^k 11 (0 (01)^{k+1})^p 0 (01)^i$ for some $0 \le i \le k$. Then $a_{2j+1}=1$ for all $j \ge k+1$, so that
\[ \begin{split}
Q(w^\clubsuit) =\ & \sum_{j=-1}^{2k+2} 2^{m-j}+2^2 + 3\sum_{j=1}^p j \sum_{\ell=0}^k 2^{m-2jk-3j-2\ell-3} + 3\sum_{\ell=0}^i (p+1) 2^{m-2(p+1)k-3(p+1)-2\ell-3}\\
<\ & 
\sum_{j=-1}^{2k+3} 2^{m-j}+2^2 < 2^{m+2}.
\end{split}\]
Since we have considered all possible values of $m$, this finishes the proof.
\end{proof}

\begin{remark}{\rm
Note that $\eta \uparrow \frac65$ corresponds to $k \to \infty$ in the previous proof, i.e., the starting block $1(10)^k110$ becomes arbitrarily large. The proof also shows that in that case $\mu_\eta \big( \big[ -\frac12, \frac12 \big] \big) \uparrow \frac23$. 
}\end{remark}

\section*{Appendix A}
In this section we derive the formula for the invariant density (\ref{q:density}) from results from \cite{Kop90}, which apply to maps $F:[0,1]\to [0,1]$ that have $F(0), F(1) \in \{0,1\}$. Let the map $F:[0,1] \to [0,1]$ and the extended version of $S_\eta$, $\bar S_\eta:[-\eta, \eta] \to [-\eta, \eta]$, be given by
\[ F (x)= \left\{
\begin{array}{ll}
2x, & \text{if }x \in \big[0, \frac12-\frac1{4\eta}\big),\\
\\
2x - \frac12, & \text{if }x \in \big[ \frac12 - \frac1{4\eta}, \frac12 + \frac1{4\eta} \big],\\
\\
2x-1, & \text{if }x \in \big( \frac12 + \frac1{4\eta} , 1\big],
\end{array}\right.
\quad \text{ and } \quad
\bar S_\eta (x) = \left\{
\begin{array}{ll}
2x+\eta, & \text{if }x \in \big[-\eta, -\frac12\big),\\
\\
2x , & \text{if }x \in \big[ -\frac12 , \frac12 \big],\\
\\
2x-\eta, & \text{if }x \in \big( \frac12 , \eta\big].
\end{array}\right.
\]
Then $F$ is conjugate to $\bar S_\eta$ with conjugacy $\phi: [-\eta, \eta] \to [0,1], \, x \mapsto \frac{x}{2\eta} + \frac12$. The critical points of $F$ are $\phi\big( \frac12 \big) = \frac12 + \frac1{4\eta}$ and $\phi\big( -\frac12 \big) = \frac12 - \frac1{4\eta}$. We will calculate the invariant density for $F$.

\medskip
For $F$ we have $F (1-x) =-1-F(x)$. Now, using the notation from \cite{Kop90}, define the points $a_1,a_2,b_1$ and $b_2$, by
\begin{align*}
a_1 &= 2\Big(\frac12 - \frac1{4\eta} \Big) = 1-\frac1{2\eta}, \quad & b_1 &= 2\Big(\frac12 - \frac1{4\eta} \Big) - \frac12 = \frac12-\frac1{2\eta},\\
a_2 &= 2\Big(\frac12 + \frac1{4\eta} \Big) - \frac12 = \frac12 + \frac1{2\eta} = 1-b_1, \quad & b_2 &= 2\Big(\frac12 + \frac1{4\eta} \Big)-1 = \frac1{2\eta} = 1-a_1. 
\end{align*}
These are the images of the critical points. The critical points divide the unit interval into three pieces, called $I_1$, $I_2$ and $I_3$ in \cite{Kop90}, so
$I_1 = \big[0, \frac12-\frac1{4\eta}\big)$, $I_2=\big[ \frac12 - \frac1{4\eta}, \frac12 + \frac1{4\eta} \big] = 1-I_2$ and $I_3 = \big( \frac12 + \frac1{4\eta} , 1\big]= 1-I_1$. Define
\[ KI_n (y) = \sum_{t \ge 0} \frac{1}{2^{t+1}} 1_{I_n} (F^t (y)).\]
Then
\begin{eqnarray*}
KI_1 (a_1) = KI_3 (b_2), & KI_2 (a_1) = KI_2 (b_2), & KI_3 (a_1) = KI_1 (b_2),\\
KI_1 (a_2) = KI_3 (b_1), & KI_2 (a_2) = KI_2 (b_1), & KI_3 (a_2) = KI_1 (b_1).\\
\end{eqnarray*}
Now define a $3 \times 2$ matrix $M = ( \mu_{i,j} )$ with entries
\begin{eqnarray*}
\mu_{1,1} &=& \frac12 + \frac12 KI_1(a_1) - \frac12 KI_1(b_1) = \frac12 + \frac12 KI_3(a_2) - \frac12 KI_3(b_2) = - \mu_{3,2},\\
\mu_{2,1} &=& -\frac12 + \frac12 KI_2(a_1) - \frac12 KI_2(b_1) = -\frac12 + \frac12 KI_2(b_2) - \frac12 KI_2(a_2) = - \mu_{2,2},\\
\mu_{3,1} &=& \frac12 KI_3(a_1) - \frac12 KI_3(b_1) = \frac12 KI_1(b_2) - \frac12 KI_1(a_2) = - \mu_{1,2}.
\end{eqnarray*}
So the matrix $M$ has the following form:
\[ M= \begin{pmatrix} a & -c\\ b &-b\\ c & -a \end{pmatrix}.\]
\cite[Lemma 1]{Kop90} gives us two other relations: For each $y$,
\begin{equation}\label{q:k7}
\frac{1}{2} KI_2(y)+KI_3(y) = y \quad \text{ and } \quad KI_1(y) + KI_2(y) +KI_3 (y) = 1.
\end{equation}
From (\ref{q:k7}) we can derive the following two relations. For $j=1,2$,
\begin{eqnarray}
\frac{1}{2} \mu_{2,j}+\mu_{3,j} &=& -\frac14 + \frac14 KI_2(a_j) -\frac14 KI_2(b_j) + \frac12 KI_3(a_j)-\frac12 KI_3(b_j) \label{q:k12} \\
&=& -\frac14 +\frac12a_j - \frac12 b_j = 0,\nonumber \\
\mu_{1,j} + \mu_{2,j} +\mu_{3,j} & = &0. \label{q:k13}
\end{eqnarray}
From (\ref{q:k12}) we get that $a=c=-\frac{b}{2}$, so the matrix $M$ becomes
\[ M= \begin{pmatrix} -\frac{b}{2} & \frac{b}{2}\\ b &-b\\  -\frac{b}{2} & \frac{b}{2}\end{pmatrix}.\]
There is a one-to-one correspondence between the solutions of the equation $M \cdot \gamma =0$ and the space of invariant measures. Since $S_\eta$ has a unique absolutely continuous invariant probability measure, $M$ has at least one non-zero entry. Any vector $\gamma$ such that $M \cdot \gamma =0$ satisfies $\gamma_1 = \gamma_2$.

\vskip .2cm
According to \cite[Theorem 1]{Kop90} a density of the invariant measures for $F$ is given by
\begin{eqnarray*}
h(x) &=& \frac12 \left( 1_{[0,a_1)}(x) - 1_{[0, b_1)}(x) + \sum_{n \ge 0} \frac{1}{2^{n+1}} (1_{[0, F^{n+1} a_1)}(x) - 1_{[0, F^{n+1}b_1)}(x)) \right)\\
&& \quad + \frac12 \left( 1_{[0,a_2)}(x) - 1_{[0, b_2)}(x) + \sum_{n \ge 0} \frac{1}{2^{n+1}} (1_{[0, F^{n+1} a_2)}(x) - 1_{[0, F^{n+1}b_2)}(x)) \right)\\
 &=& \sum_{n \ge 0} \frac{1}{2^{n+1}} \Big(1_{[0, F^n a_1)}(x) - 1_{[0, F^n b_1)}(x) + 1_{[0, F^n a_2)}(x) - 1_{[0, F^n b_2)}(x) \Big)\\
 &=& \sum_{n \ge 0} \frac{1}{2^{n+1}} \Big(1_{[F^n b_1, F^n a_1)}(x) - 1_{[F^n a_1, F^n b_1)}(x) + 1_{[F^n b_2, F^n a_2)}(x) - 1_{[F^n a_2, F^n b_2)}(x) \Big).
\end{eqnarray*}
When translated back to the map $S_\eta$, this gives the formula from (\ref{q:density}).

\section*{Appendix B}

The content of this section is solely due to D. Kong \cite{K18},  it gives the proof of Proposition \ref{equivalenceDerong}.   For ease of notation we denote $\overline{w_1\ldots w_n}=(1-w_1)\ldots (1-w_n)$.
\begin{proof} (Proof of Proposition \ref{equivalenceDerong}) Suppose $w=w_1\cdots w_m$ is primitive. To show that $w$ is admissible, by Definition~\ref{d:primitive}(ii) it suffices to show that  
\begin{equation}\label{eq:1}
\overline{w_1\ldots w_{m-k}}\prec  w_{k+1}\ldots w_m \quad\textrm{for all}\quad 0\le k\le m-1. 
\end{equation}
Since $w_m=1$, (\ref{eq:1}) holds for $k=0$. Suppose (\ref{eq:1}) fails for some  minimal $1\le k\le m-1$. Then $w_{k+1}\ldots w_m\preceq \overline{w_1\ldots w_{m-k}}$. By the minimality of $k$ we have $w_k=1$.   Then by using $w_m=1$ it follows that 
\[
 w_1\ldots w_k0^{\infty}\prec w_1\ldots w_m0^{\infty}\preceq w_1\ldots w_k\overline{w_1\ldots w_{m-k}}0^{\infty}\prec w_1\ldots w_k\overline{w_1\ldots w_{k-1}}1 0^{\infty}.
\]
which contradicts primitivity. Hence, $w_1\ldots w_m$ is an admissible word. Conversely, assume $w_1\cdots w_m$ is admissible.  Then, for any $0\le k\le m-1$ we have 
\begin{equation}\label{eq:2}
\overline{w_1\ldots w_{m-k}}\prec w_{k+1}\ldots w_m\preceq w_1\ldots w_{m-k}.
\end{equation}
So to prove the result, we only need to check conditions (i) and (iii) of Definition \ref{d:primitive}. Taking $k=m-1$ in (\ref{eq:2}) it follows that $w_1=w_m=1$. If $m=2$, then $w_1=w_2=w_m=1$. Let $m\ge 3$ and assume $w_2=0$, then by (\ref{eq:2}) it follows that 
\[
w_1\ldots w_m=(10)^{\ell} 1\quad\textrm{with}\quad \ell\ge 1,
\]
leading to a contradiction with (\ref{eq:2}) since
$
w_2\ldots w_m=(01)^{\ell}=\overline{w_1\ldots w_{m-1}}.
$
Therefore, $w_2=1$. It remains to prove condition (iii) Definition \ref{d:primitive}. Suppose it fails and then there exists a word $b_1\ldots b_j$  such that 
\begin{equation}
\label{eq:3}
b_1\ldots b_j0^{\infty}\prec w_1\ldots w_m0^{\infty}\prec b_1\ldots b_j\overline{b_1\ldots b_{j-1}}1 0^{\infty}.
\end{equation}
We can assume with no loss of generality that $b_j=1$. Then by (\ref{eq:3}) we have $j<m$ and $b_1\ldots b_j=w_1\ldots w_j$.
Furthermore, 
\begin{align*}
w_{j+1}\ldots w_m&\preceq \overline{b_1\ldots b_{m-j}}=\overline{w_1\ldots w_{m-j}}&&\textrm{if }m<2j;\\
w_{j+1}\ldots w_{2j}&\prec \overline{b_1\ldots b_{j-1}}1=\overline{w_1\ldots w_{j-1}}1&&\textrm{if }m\ge 2j.
\end{align*}
If $m<2j$, then $w_{j+1}\ldots w_m\preceq \overline{w_1\ldots w_{m-j}}$, leading to a contradiction with (\ref{eq:2}). If $m=2j$, then 
\[w_{j+1}\ldots w_{m-j}=w_{j+1}\ldots w_{2j}\preceq \overline{w_1\ldots w_j}=\overline{w_1\ldots w_{m-j}},\]
again leading to a contradiction with (\ref{eq:2}). If $m> 2j$, then $w_{j+1}\ldots w_{2j}\preceq\overline{w_1\ldots w_j}$. By (\ref{eq:2}) with $k=j$ it gives that
\[
w_1\ldots w_{2j}=w_1\ldots w_j\overline{w_1\ldots w_j}.
\]
Again by iterations of (\ref{eq:2}) we conclude that 
\[
w_1\ldots w_m=(w_1\ldots w_j\overline{w_1\ldots w_j})^{\ell} w_1\ldots w_p\quad\textrm{with}\quad \ell\ge1\textrm{ and } 1\le p<2j.
\]
Then
\[
w_{m-p-j+1}\ldots w_m=\overline{w_1\ldots w_j}w_1\ldots w_p=\overline{w_1\ldots w_{p+j}},
\]
leading to a contradiction with (\ref{eq:2}). Thus condition (iii) of Definition \ref{d:primitive}, and $w_1\ldots w_m$ is primitive.
\end{proof}

\end{document}